\documentclass[11pt,a4paper,reqno]{amsart}%article}%
\usepackage[margin=1.6in]{geometry}

\usepackage{tikz-cd} %used for drawing commutative diagrams
\usepackage{graphicx} %supports image insertion
\usepackage{adjustbox} %supports adjusting box sizes
\usepackage{enumerate} %supports list environments
\usepackage{amssymb} %provides additional mathematical symbols and fonts
\usepackage{amsmath} %provides rich typesetting capabilities for mathematical expressions in LaTeX
\usepackage{mathtools} %provides more mathematical environments and commands
\usepackage{amsfonts} %provides additional mathematical fonts
\usepackage{amsthm} %supports theorem and proof environments
\usepackage{mathrsfs} %provides additional mathematical script fonts
\usepackage{esint}
\usepackage{epigraph} 
\usepackage{calligra}
\usepackage{array}
\usepackage{color}

\usepackage{amssymb}

\usepackage{subfiles} %supports separately compiling sub-files
\usepackage{appendix} %supports appendix environment
\usepackage{pgfplots, tikz, stmaryrd}
\usepackage[english]{babel} %supports multiple languages
\usepackage[colorlinks,linkcolor=red,anchorcolor=green,citecolor=blue]{hyperref} %provides cross-referencing functionality
\hypersetup{linktocpage = true}

\usepackage{xcolor}
\hypersetup{colorlinks,linkcolor={blue!50!black},citecolor={blue!50!black},urlcolor={blue!80!black}}

\usepackage{fouriernc}
\usepackage{fancyhdr}
\usepackage{lastpage}
\usepackage{xstring}

\usepackage{amsthm}

\newenvironment{itenv*}
  {\phantomsection\par\medskip\noindent\itshape}
  {\par\medskip}

\usepackage{cleveref}
\newtheorem{defnsub}{Definition}[subsection]
\newtheorem{lemmasub}[defnsub]{Lemma}
\newtheorem{propsub}[defnsub]{Proposition}
\newtheorem{thmsub}[defnsub]{Theorem}
\newtheorem{corsub}[defnsub]{Corollary}

\newtheorem{remksub}[defnsub]{Remark}
\newtheorem{quessub}[defnsub]{Question}

\numberwithin{equation}{subsection}

\DeclareMathOperator{\Z}{\mathbb{Z}} %integers
\DeclareMathOperator{\R}{\mathbb{R}} %real numbers
\DeclareMathOperator{\C}{\mathbb{C}} %complex numbers
\renewcommand{\Re}{\mathrm{Re}} %real part
\renewcommand{\Im}{\mathrm{Im}} %imaginary part
\DeclareMathOperator{\id}{id} %identity map
\DeclareMathOperator{\tr}{tr} %trace
\DeclareMathOperator{\Hom}{Hom} %homomorphism
\DeclareMathOperator{\cO}{\mathcal{O}} %structure sheaf
\DeclareMathOperator{\sheafhom}{\mathscr{H}\text{\kern -3pt {\calligra\large om}}\,} %sheaf hom

\renewcommand{\div}{\mathrm{div}} %divergence

\DeclareMathOperator{\ch}{ch} %chern character
\DeclareMathOperator{\cD}{\mathcal{D}} %currents
\DeclareMathOperator{\coh}{Coh} %coherent sheaf category
\DeclareMathOperator{\Bl}{Bl} %blow up
\renewcommand{\P}{\mathbb{P}} %projective space
\DeclareMathOperator{\Eff}{Eff} %effective cone
\DeclareMathOperator{\Amp}{Amp} %ample cone

\usetikzlibrary{decorations.markings}
\tikzset{->-/.style={decoration={ markings, mark=at position #1 with
{\arrow{>}}},postaction={decorate}}}
\tikzset{-<-/.style={decoration={ markings, mark=at position #1 with
{\arrow{<}}},postaction={decorate}}}

\def\lp2{\mathcal O_{\mathbb P^2}(-3)}
\def\p2{\mathbb P^2}
\def\p3{\mathbb P^3}
\def\Stab{\mathrm{Stab}}
\def\coh{\mathrm{Coh}}

\def\num{\mathrm{num}}
\def\Ker{\mathrm{Ker}}

\def\gl2{\widetilde{\mathrm{GL}}^+_{2}(\mathbb R)}

\def\r{\mathrm{r}}

\def\cC{\mathcal C}
\def\cD{\mathcal D}

\def\cL{\mathcal L}

\def\cO{\mathcal O}

\def\cS{\mathcal S}

\def\cZ{\mathcal Z}
\usepackage{url}

\usetikzlibrary{matrix,positioning,decorations.markings,arrows,decorations.pathmorphing, 
backgrounds,fit,positioning,shapes.symbols,chains,shadings,fadings,calc}
\tikzset{->-/.style={decoration={ markings, mark=at position #1 with
{\arrow{>}}},postaction={decorate}}}
\tikzset{-<-/.style={decoration={ markings, mark=at position #1 with
{\arrow{<}}},postaction={decorate}}}
\DeclareMathAlphabet{\mathcal}{OMS}{cmsy}{m}{n}

\begin{document}
%=========================================================
\title[Relative stability conditions]{Relative stability conditions on triangulated categories}

\author{Bowen Liu}
\address{Department of Mathematical Sciences, Tsinghua University, 100084 Beijing, China}
\email{liubw22@mails.tsinghua.edu.cn}

\author{Dongjian Wu}
\address{Department of Mathematical Sciences, Tsinghua University, 100084 Beijing, China}
\email{wdj20@mails.tsinghua.edu.cn}

\date{}
%\keywords{Relative (Bridgeland) stability conditions, De}

%=========================================================

\maketitle
%\tableofcontents
%\addtocontents{toc}{\setcounter{tocdepth}{1}}

%=========================================================

\begin{abstract}
We introduce the notion of relative stability conditions on triangulated categories with respect to left admissible subcategories, based on \cite{MR2373143}, and demonstrate the deformation of relative stability conditions via the deformation of gluing stability conditions in \cite{Collins2009GluingSC}. The motivation for this concept stems from the discussions in \cite{MR4479717} concerning the relationship between Bridgeland stability and the existence of the deformed Hermitian-Yang-Mills metrics on line bundles.

%Subsequently, we propose a conjecture on relative stability and the solvability of the deformed Hermitian-Yang-Mills equations.
\end{abstract}

\maketitle
\tableofcontents
\setcounter{section}{1}

%\addtocontents{toc}{\setcounter{tocdepth}{1}}

\setlength\parindent{0pt}
\setlength{\parskip}{5pt}

%=========================================================
\section*{Introduction}
%=========================================================
%=========================================================
%\subsection*{Notations}
%=========================================================

\subsection{Relative stability conditions}
The notion of stability conditions on triangulated categories was initially introduced by Bridgeland \cite{MR2373143}, drawing inspiration from the fields of string theory and mirror symmetry. 
This concepts extends the classical notion of slope stability for vector bundles. The idea of relative stability conditions on partially wrapped Fukaya categories of marked surfaces was first developed in \cite{Takeda2018RelativeSC}, with the aim of reducing the calculation of stability conditions on Fukaya categories of fully stopped surfaces into basic cases through the process of cutting and gluing. 

In this paper, we introduce the notion of relative stability conditions on triangulated categories with respect to left admissible subcategories. Both concepts utilize the gluing technique. Our notion of spaces of relative stability conditions possesses the deformation property and is of finite dimension. However, the spaces described in \cite{Takeda2018RelativeSC} are infinite-dimensional. 

Let $\mathscr{D}$ be a triangulated category with a left admissible subcategory $\mathscr{D}_1$ and $\mathscr{D}=\langle\mathscr{D}_1,\mathscr{D}_2:={^{\perp}\mathscr{D}_1}\rangle$ be the semiorthogonal decomposition defined in 
\cite{Bondal1995SemiorthogonalDF, MR1957019}. Throughout this paper, we maintain the assumption that the semiorthogonal decomposition under consideration is \emph{polarizable}, meaning that each factor admits a stability condition. Denote by $i^{\ast}_1$ the left adjoint functor to the inclusion $i_1\colon \mathscr{D}_0\hookrightarrow\mathscr{D}$ and $i_2^{!}$  the right adjoint functor to the inclusion $i_2\colon\mathscr{D}_2\hookrightarrow\mathscr{D}$. We consider the space of stability conditions on $\mathscr{D}_i$ with respect to $(\Lambda_i,v_i)$. We define the relative stability conditions as follows:
\begin{defnsub}
\label{def-rel-1}
A relative stability condition $\sigma_{\r}=(Z,\mathscr{P}_1)$ on $\mathscr{D}$ with respect to a left admissible subcategory $\mathscr{D}_1$ is defined by a group homomorphism $Z\colon \Lambda\to\mathbb C$ and a slicing $\mathscr{P}_1$ of $\mathscr{D}_1$ that fulfills the following conditions:
\begin{enumerate}[$(1)$]
\item $(Z,\mathscr{P}_1)$ can be \emph{relatively extended} to a stability condition ${\sigma}=(Z,{\mathscr{P}})$ on $\mathscr{D}$, i.e. $\mathscr{P}_1(\phi)\subset{\mathscr{P}}(\phi)$ for any $\phi\in\mathbb R$, and $(Z_2,\mathscr{P}_2):=(Z\vert_{\Lambda_2},\mathscr{P}\cap\mathscr{D}_2)\in\Stab(\mathscr{D}_2)$; 
\item There exist $0<\epsilon<\delta<1$ such that for any heart $\mathscr{A}_i\in\mathscr{P}_i((\delta-\epsilon,1+\delta+\epsilon))$ of a bounded $t$-structure of $\mathscr{D}_i$,
\[
\langle\mathscr{A}_2,\mathscr{A}_1\rangle:=\{X\in\mathscr{D}\mid i_1^{\ast}(X)\in\mathscr{A}_1, i^{!}_{2}(X)\in\mathscr{A}_2\}
\]
constitutes a heart of a bounded $t$-structure in $\mathscr{D}$.
\end{enumerate}
\end{defnsub}
Note that Condition (2) in \Cref{def-rel-1} is inherently satisfied if $\mathscr{P}_2$ is of finite length and has finite number of simple objects. Denote by $\Stab(\mathscr{D},\mathscr{D}_1)$ the space of relative stability conditions on $\mathscr{D}$ with respect to $\mathscr{D}_1$. Based on \Cref{redundant} and \Cref{eredundant}, in the cases of the bounded derived categories of coherent sheaves on any smooth projective complex variety, as well as to the bounded derived category of modules over a finite-dimensional algebra with finite global dimension. the Condition (2) in \Cref{def-rel-1} is not required.

In Section \ref{sec: deform}, we demonstrate the following deformation property of spaces of relative stability conditions via the deformation of gluing stability conditions in \cite{Collins2009GluingSC}, which extends the result in \cite{MR2373143}. 
\begin{thmsub}[Theorem \ref{rdeformation}]
Let $\mathscr{D}$ be a triangulated category with a left admissible subcategory $\mathscr{D}_1\subset\mathscr{D}$. The space of relative stability conditions $\Stab(\mathscr{D},\mathscr{D}_1)$ has the structure of a complex manifold, and the map 
\[
\cZ_1\colon\Stab(\mathscr{D},\mathscr{D}_1)\to\Hom_{\mathbb Z}(\Lambda,\mathbb C)
\]
that sends a relative stability condition to its relative central charge is a local isomorphism.
\end{thmsub}
In other words, Theorem \ref{rdeformation} allows us to deform a relative stability condition $(Z,\mathscr{P}_1)$ (uniquely) by deforming $Z$.  If $\mathscr{D}_1=\mathscr{D}$, $\Stab(\mathscr{D},\mathscr{D}_1)$ is equivalent to $\Stab(\mathscr{D})$, which indicates that stability conditions are a specialized form of relative stability conditions. In addition, we introduce some initial examples of spaces of relative stability conditions on triangulated categories in Section \ref{initial-example}.

%=========================================================
\subsection{Motivations from dHYM equations}
It is a fundamental principle that has directed a significant portion of research in complex geometry since the late 20-th century: stable objects in algebraic geometry should correspond to extremal objects in differential geometry. This philosophy was inspired by the relations between the slope stability and the existence of Hermitian-Einstein metrics, which has been explored extensively in studies such as \cite{MR0184252,MR0710055,MR0765366,MR0861491,MR0944577}. The large volume limit of Bridgeland stability is understood to be linked to slope stability. As a result, it is intriguing to examine the relations between Bridgeland stability and the existence of dHYM on line bundles. The work in \cite{MR4479717} demonstrates that on $\Bl_p\P^2$, every line bundle that admits a dHYM metric with for a given ample $\R$-divisor $\omega$ is Bridgeland stable with respect to a stability condition associated with $\omega$. However, the reverse implication does not hold. This result was later generalized in \cite{collins2023stabilitylinebundlesdeformed} for Weierstrass elliptic K3 surfaces. In a broader framework, \cite[Question 3.3]{collins2023stabilitylinebundlesdeformed} proposes a conjecture suggesting that twisted ampleness is a sufficient condition for Bridgeland stable on any smooth projective complex surface. In order to establish the converse, one straightforward approach is to contract the hearts of stability conditions, which leads to the development of the concept of relative stability conditions in this paper. 

The question concerning the relation between relative stability and the solvability of dHYM equation for any smooth projective complex variety is formulated as follows:
\begin{quessub}[Question \ref{main-question}]
Let $X$ be a smooth projective complex variety. Given an ample $\R$-divisor $\omega$ and a $\R$-divisor $B$. Determine admissible subcategories $0\ne\mathscr{D}_1$ of $\mathscr{D}^b(X)$ for which a line bundle $\mathcal{L}\in\mathrm{Coh}(X)$ is relatively stable with respect to $\sigma_{\r,\omega,B}=(Z_{\omega,B},\mathscr{P}_1:=\mathscr{P}_{\omega,B}\cap\mathscr{D}_1)\in\Stab(X,\mathscr{D}_1)$
if and only if $\cL\in\mathscr{P}_1((0,1])$ admits a dHYM metric with respect to $(\omega, B)$.
\end{quessub}

In Proposition \ref{prop: surf-with-one}, we establish that Question \ref{main-question} is resolved by any admissible subcategories for the case of projective curves with positive genus and projective surfaces without negative self-intersection curve. In Proposition \ref{prop: possible admissible subcategory for question on Hirzebruch surface}, we identify certain admissible subcategories generated by exceptional objects that resolve Question \ref{main-question} for Hirzebruch surfaces. Furthermore, Proposition \ref{prop: counter-example for large volume case} shows that there is a discrepancy between Bridgeland stability and the dHYM metric on Hirzebruch surfaces, suggesting that the 
consideration of relative stability conditions (or Question \ref{main-question}) is reasonable for exploring the relation between Bridgeland stability and the existence of the dHYM metric on line bundles.

%=========================================================
\subsection{Contents}
In Section \ref{sec: pre}, we review the basic facts of Bridgeland stability conditions on triangulated categories, as well as the semiorthogonal decomposition. In Section \ref{sec: relative-stab}, we introduce the notion of relative (Bridgeland) stability conditions on triangulated categories with respect to left admissible subcategories, and then we illustrate the deformation of these relative stability conditions. Moreover, we will provide some examples. In Section \ref{sec: relstab-dHYM}, we offer a brief overview of the background on dHYM equations. We then propose a question concerning the relations between the relative stability and solvability of dHYM equations. Afterward, we examine the question through some cases.

%=========================================================
\subsection*{Acknowledgement}
Bowen Liu expresses gratitude to his supervisor Chenglong Yu for his consistent support, and he also extends thanks to Mao Sheng for guiding him into the study of the relations between dHYM equations and Bridgeland stability conditions.
Dongjian Wu is grateful for his supervisor Yu Qiu for his continuous support and patience throughout his research. He is also thankful to Fabian Haiden for valuable insights and discussions. In addition, he is appreciative of the hospitality offered by the centre for QM at SDU.

%=========================================================

\section{Preliminaries}
\label{sec: pre}
In this section, we will review some fundamental results of Bridgeland stability conditions on triangulated categories in \cite{MR2373143}, as well as semiorthogonal decomposition as presented in \cite{Bondal1995SemiorthogonalDF, MR1957019}.  Throughout the paper, we assume that the Grothendieck group of a triangulated category $\mathscr{D}$, denoted by $K(\mathscr{D})$, is finitely generated.

\subsection{Bridgeland stability conditions}
\label{sec: stab}
Let $\mathscr{D}$ be a triangulated category, for which we fix a finite rank lattice $\Lambda$ with a surjective group homomorphism $v\colon K(\mathscr{D})\twoheadrightarrow\Lambda$ and a norm $\Vert\cdot\Vert$ on $\Lambda_{\mathbb R}$.

\begin{defnsub}
A \emph{slicing} $\mathscr{P}$ of a triangulated category $\mathscr{D}$ consists of full additive subcategories $\mathscr{P}(\phi)\subset\mathscr{D}$ for each $\phi\in\mathbb R$ satisfying the following axioms:
\begin{enumerate}[$(a)$]
\item for all $\phi\in \mathbb R$, $\mathscr{P}(\phi +1) = \mathscr{P}(\phi)[1]$;
\item if $\phi_1>\phi_2$ and $A_i\in \mathscr{P}(\phi_i)\,,i=1,2$, then ${\rm Hom}_{\mathscr{D}}(A_1, A_2) = 0$;
\item for $0\neq E\in \mathscr{D}$, there is a finite sequence of real numbers
\[
\phi_1>\phi_2>\dots>\phi_m
\]
and a collection of triangles called \emph{Harder-Narasimhan filtration}
\begin{center}
% https\colon//tikzcd.yichuanshen.de/#N4Igdg9gJgpgziAXAbVABwnAlgFyxMJZABgBpiBdUkANwEMAbAVxiRGIF4BRAfWJAC+pdJlz5CKAEzkqtRiza8AjIOEgM2PASIBmGdXrNWiEAB1TUCDgRCRm8UQAs+uUcU9gAWwC0SgarsxbRQANhdDBRNeT24A9VEtCWQlUiVZCOMQAEEeFVt4+2DkAFZU9PlMnM9BWRgoAHN4IlAAMwAnCGrEMhAcCCQU10iwJgYGagY6ACMYBgAFBIcTBhgWnDj2zoHqPqRSoeMRsYnp2YXCiRAVtY2OrsHdxGkDpCPxq9P5xeCr1fX8zZdZ6PPQvRBvE4zL4XNjXf5qQFIUGPZxgiEfKHnIKXOG3LaIVGPMJo0bvSaY744v54rr7R49DKvUmQs6U2HUiZYMCZKB0OAACzqNKQxJBBgqTOOGNZMOWHKuXJ5VhwQoEFAEQA
\begin{tikzcd}
0=E_0 \arrow[rr] &                        & E_1 \arrow[ld] \arrow[r] & \dots \arrow[r] & E_{m-1} \arrow[rr] &                        & E_m=E \arrow[ld] \\
                 & A_1 \arrow[lu, dashed] &                          &                 &                    & A_m \arrow[lu, dotted] &                 
\end{tikzcd}
\end{center}
with $A_i\in\mathscr{P}(\phi_i)$ for all $1\leq i\leq m$.
\end{enumerate}
\end{defnsub} 
Let $E$ be a nonzero object in $\mathscr{D}$ that admits a Harder-Narasimhan filtration as described in axiom (c). We associate two numbers with $E\colon \phi^+_{\sigma}(E):=\phi_1$ and $\phi^-_{\sigma}(E):=\phi_m$, where $\phi_1$ and $\phi_m$ are the phases from axiom (c). An object $E\in \mathscr{P}(\phi)$ for some $\phi\in \mathbb R$ is called semistable, and in such a case, $\phi = \phi^{\pm}_{\sigma}(E)$. Moreover, if $E$ is a simple object in $\mathscr{P}(\phi)$, it is said to be stable. We define $\mathscr{P}(I)$ for an interval $I$ in $\mathbb R$ as
 \[
 \mathscr{P}(I) = \{E\in \mathscr{D}\ |\ \phi^{\pm}_{\sigma}(E)\in I\}\cup\{0\}.
 \]
Consequently, for any $\phi\in \mathbb R$, both $\mathscr{P}[\phi, \infty)$ and $\mathscr{P}(\phi, \infty)$ are $t$-structures in $\mathscr{D}$ with hearts $\mathscr{P}((\phi,\phi+1])$ and $\mathscr{P}([\phi,\phi+1))$, respectively. According to \cite{MR2373143}, each category $\mathscr{P}(\phi)\subset\mathscr{D}$ is abelian and $\mathscr{P}(I)\subset\mathscr{D}$ is quasi-abelian for any interval $I\subset\mathbb R$ of length $<1$.

\begin{defnsub}[\protect{\cite[Definition 5.1]{MR2373143}}]
A \emph{Bridgeland pre-stability condition} $\sigma=(Z, \mathscr{P})$ on $\mathscr{D}$ with respect to $(\Lambda,v)$ is characterized by a group homomorphism $Z\colon  \Lambda\to \mathbb C$, termed the \emph{central charge} and a slicing $\mathscr{P}$ of $\mathscr{D}$ such that if $0\ne E\in\mathscr{P}(\phi)$ then $Z(v(E))=m(E)\mathrm{exp}(\sqrt{-1}\pi\phi)$ for some $m(E)\in\mathbb R_{>0}$.
\end{defnsub}

An object $E\in\mathscr{D}$ is called $\sigma$-stable (resp. $\sigma$-semistable) if $E$ is stable (resp. semistale) with respect to the slicing $\mathscr{P}$. We now restrict our attention to Bridgeland pre-stability conditions that fulfill the \emph{support property} (\cite{kontsevich2008stabilitystructuresmotivicdonaldsonthomas}). 

\begin{defnsub}
\label{def: support}
A Bridgeland pre-stability condition $\sigma=(Z,\mathscr{P})$ satisfies the support property with repsect to $(\Lambda,v)$ if there exists a constant $C>0$ such that for all $\sigma$-semistable objects $0\ne E\in\mathscr{D}$, we have 
\[
\Vert v(E)\Vert\le C\vert Z(v(E))\vert.
\]
\end{defnsub}

There is an equivalent formulation of support property (\cite[Section 2.1]{kontsevich2008stabilitystructuresmotivicdonaldsonthomas}): $\sigma=(Z,\mathscr{P})$ satisfies the support property with respect to $(\Lambda,v)$ if there  exists a quadratic form $Q$ on $\Lambda_{\mathbb R}$ such that 
 \begin{enumerate}[$(a)$]
         \item $\Ker\,Z$ is negative definite with respect to $Q$;
         \item Every $\sigma$-semistable object $E\in\mathscr{D}$ satisfies $Q(v(E))\ge0$. 
 \end{enumerate}
%Note that the notion of a full locally-finite stability condition in \cite[Definition 4.2]{MR2373143} is also equivalent to the support property as described above.  

A Bridgeland pre-stability condition that satisfies the support property is called a \emph{Bridgeland stability condition}. If we take $\Lambda$ to be the numerical Grothendieck group $K_{\num}(\mathscr{D})$, and $v$ the natural projection, $\sigma$ is called a \emph{numerical Bridgeland stability condition}. Unless there is a specific need to emphasis $(\Lambda,v)$, we will denote the set of stability conditions with respect to $(\Lambda,\lambda)$ simply by $\mathrm{Stab}(\mathscr{D})$. A stability condition $\sigma$ is said to be \emph{locally finite} if for any $\phi$ there exists an $\epsilon>0$ such that $\mathscr{P}((\phi-\epsilon,\phi+\epsilon))$ is finite length. 

As described in \cite{MR2373143}, $\mathrm{Stab}(\mathscr{D})$ possesses a natural topology induced by the generalized metric
\[
d(\sigma_1,\sigma_2)=\sup\limits_{0\ne E\in\mathscr{D}}\left\{|\phi_{\sigma_2}^-(E)-\phi_{\sigma_1}^-|, |\phi^+_{\sigma_2}(E)-\phi^+_{\sigma_1}(E)|, \left|\mathrm{log}\frac{m_{\sigma_2}(E)}{m_{\sigma_1}(E)} \right| \right\}.
\]
The main theorem in \cite{MR2373143} asserts that $\Stab(\mathscr{D})$ possesses a complex structure:
\begin{thmsub}[\protect{\cite[Theorem 1.2]{MR2373143}}]
\label{deformation}
The space of stability conditions $\mathrm{Stab}(\mathscr{D})$ has the structure of a complex manifold of dimension $\mathrm{rk}(\Lambda)$, and the map
\begin{equation*}
\cZ\colon \mathrm{Stab}(\mathscr{D})\to\mathrm{Hom}_{\mathbb Z}(\Lambda,\mathbb C)
\end{equation*}
that sends a stability condition to its central charge is a local isomorphism.
\end{thmsub}

 A connected component of $\Stab(\mathscr{D})$ is called $\emph{full}$ if it has the maximal dimension, meaning that its image under $\cZ$ is equal to $\Hom_{\mathbb Z}(\Lambda,\mathbb C)$. A stability condition is \emph{full} if it lies in a full component.

\begin{defnsub}[\protect{\cite[Definition 2.1, 2.2]{MR2373143}}]
A \emph{stability function} on an abelian category $\mathscr{A}$ is a group homomorphism $Z\colon K(\mathscr{A})\to\mathbb C$ such that for all $0\ne E\in\mathscr{A}$ the complex number $Z(E)$ lies in the strict upper half-plane, i.e.
$$
\begin{aligned}
Z(E)&=m(E)\cdot\mathrm{exp}(\sqrt{-1}\phi(E))\in\mathbb H_+\\
&:=\{m\cdot\mathrm{exp}(\sqrt{-1}\pi\phi)\mid m\in\mathbb R_{>0}, \phi\in(0,1]\}\subset\mathbb C,
\end{aligned}
$$
and $\phi(E)$ is called the \emph{phase} of $E$.
\end{defnsub}

An object $0\ne E\in\mathscr{A}$ is called $Z$-\emph{semistable} if every non-zero object $A\subset E$ satisfies $\phi(A)\le\phi(E)$. According to \cite[Proposition 5.3]{MR2373143}, a stability condition $\sigma=(Z,\mathscr{P})\in\Stab(\mathscr{D})$ can be equivalently defined in terms of a pair $(Z_{\mathscr{A}},\mathscr{A})$, where $\mathscr{A}$ is the heart of a bounded $t$-structure on $\mathscr{D}$ (\cite{MR0751966}) and $Z_{\mathscr{A}}$ is a \emph{stability function} on $K(\mathscr{A})$ that satisfies the Harder-Narasimhan property and the support property. The \emph{mass} of an object $0\ne E\in\mathscr{D}$ is defined as $m_{\sigma}(E):=\Sigma_i\vert Z(A_i)\vert$, where $\{A_i\}$ are the semistable factors of $E$ from the Harder-Narasimhan filtration of $E$. 

%\subsection{Group actions on stability conditions}
%\label{sec: G-action}
Let $\mathrm{Aut}(\mathscr{D})$ denote the group of automorphisms of $\mathscr{D}$, $\mathrm{GL}^+_2(\mathbb R)$ be the group of elements in $\mathrm{GL}_2(\mathbb R)$ with positive determinant and let $\widetilde{\mathrm{GL}}_2^+(\mathbb R)$ be the universal cover of $\mathrm{GL}^+_2(\mathbb R)$. There exists a left action of $\mathrm{Aut}(\mathscr{D})$ and a right action of the group $\widetilde{\mathrm{GL}}_2^+(\mathbb R)$ on $\mathrm{Stab}_{\Lambda}(\mathscr{D})$ as described in \cite{MR2373143}.  
For any $T\in\mathrm{Aut}(\mathscr{D})$ and stability condition $\sigma=(Z,\mathscr{P})$, we define 
\[
T\sigma=(Z',\mathscr{P}'),\quad Z'=ZT^{-1},\quad \mathscr{P}'_{\phi}=T\mathscr{P}_{\phi}.
\]

For any element $g=(T,f)\in\widetilde{\mathrm{GL}}_2^+(\mathbb R)$, we define
\[
\sigma[g]=(Z[g],\mathscr{P}[g]), \quad Z[g]=T^{-1}Z,\quad \mathscr{P}[g]_{\phi}=\mathscr{P}_{f(\phi)}. 
\]
In particular, for $a+\sqrt{-1}b\in\mathbb C\subset\widetilde{\mathrm{GL}}_2^+(\mathbb R)$, we have
\[
Z[a+\sqrt{-1}b]=e^{-\sqrt{-1}\pi a+\pi b}Z, \quad \mathscr{P}[a+\sqrt{-1}b]_{\phi}=\mathscr{P}_{\phi+a}.
\]

\subsection{Reviews of semiorthogonal decompositions}
Recall that a \emph{semiorthogonal decomposition} of a triangulated category $\mathscr{D}$ is a collection $\mathscr{D}_1,\dots,\mathscr{D}_n$ of full triangulated subcategories such that 
\begin{enumerate}[$(1)$]
\item  $\Hom_{\mathscr{D}}(\mathscr{D}_i,\mathscr{D}_j)=0$ for all $1\le j<i\le n$;
\item $\mathscr{D}=\mathrm{thick}\langle\mathscr{D}_1,\dots,\mathscr{D}_n\rangle$, i.e. $\mathscr{D}$ is the smallest triangulated subcategory containing $\mathscr{D}_1,\dots,\mathscr{D}_n$.
\end{enumerate}
%===================================================================
We will use the notation $\mathscr{D}=\langle\mathscr{D}_1,\dots,\mathscr{D}_n\rangle$ to represent a semiorthogonal decomposition of $\mathscr{D}$ with components $\mathscr{D}_1,\dots,\mathscr{D}_n$. A full triangulated subcategory $\mathscr{D}_1\subset\mathscr{D}$ is called \emph{left (right) admissible} if its embedding functor $i\colon\mathscr{D}_1\to\mathscr{D}$ possesses left (right) adjoint functors $i^{\ast} (i^{!})\colon\mathscr{D}\to\mathscr{D}_1$. A left admissible subcategory $\mathscr{D}_1\subset\mathscr{D}$ induces a semiorthogonal decomposition $\mathscr{D}=\langle\mathscr{D}_1, ^{\perp}\mathscr{D}_1\rangle$, where
\[
^{\perp}{\mathscr{D}_1}:={E\in\mathscr{D}\mid\Hom(E,\mathscr{D}_1[n])=0 \text{ for all } n\in\mathbb Z}
\]
denotes the \emph{left orthogonal} to $\mathscr{D}_1$ in $\mathscr{D}$.
A right admissible subcategory $\mathscr{D}_1\subset\mathscr{D}$ leads to a semiorthogonal decomposition $\mathscr{D}=\langle\mathscr{D}_1^{\perp},\mathscr{D}_1\rangle$, with the \emph{right orthogonal} $\mathscr{D}_1^{\perp}$ defined by
\[
\mathscr{D}_1^{\perp}:={E\in\mathscr{D}\mid\Hom(\mathscr{D}_1[n],E)=0 \text{ for all } n\in\mathbb Z}.
\]
A full triangulated category is called \emph{admissible} if it is both left and right admissible. Given a semiorthogonal collection of admissible subcategories $\mathscr{D}_1,\dots,\mathscr{D}_m$ in $\mathscr{D}$, for each $0\le k\le m$, there exists a semiorthogonal decomposition
\[
\mathscr{D}=\langle\mathscr{D}_1,\dots,\mathscr{D}_k, ^{\perp}\langle\mathscr{D}_1,\dots,\mathscr{D}_k\rangle\cap\langle\mathscr{D}_{k+1},\dots,\mathscr{D}_m\rangle^{\perp},\mathscr{D}_{k+1},\dots,\mathscr{D}_m\rangle.
\]
The straightforward example of an admissible subcategory is one generated by an \emph{exceptional} object. An object $E$ is \emph{exceptional} if $\Hom(E,E)=\C$ and $\Hom(E,E[t])=0$ for $t\ne0$. An \emph{exceptional collection} is a collection of exceptional objects ${E_1, E_2, \dots, E_m}$ such that $\Hom(E_i,E_j[t])=0$ for all $i>j$ and all $t\in\mathbb Z$. An exceptional collection in $\mathscr{D}$ yields a semiorthogonal decomposition
\[
\mathscr{D}=\langle\mathscr{A},E_1,\dots,E_m\rangle \quad \text{ with } \mathscr{A}=\langle E_1,\dots,E_m\rangle^{\perp},
\]
where $E_i$ represents the subcategory generated by the same exceptional object. If $\mathscr{A}=0$, then the exceptional collection ${E_1,\dots,E_m}$ is referred to as \emph{full}.
%===================================================================

\section{The space of relative stability conditions}
\label{sec: relative-stab}
In this section, we introduce the definition of relative (Bridgeland) stability conditions on triangulated categories with respect to left admissible subcategories based on \cite{MR2373143}. Subsequently, we illustrate the deformation of these relative stability conditions by using the deformation of gluing stability conditions (\cite{Collins2009GluingSC}). Finally, we will provide some initial examples.

\subsection{Relative stability conditions}
Let $\mathscr{D}$ be a triangulated category and $\mathscr{D}_1$ a left admissible subcategory of $\mathscr{D}$. Consider the semiorthogonal decompositions $\mathscr{D}=\langle\mathscr{D}_1, \mathscr{D}_2:=^{\perp}\mathscr{D}_1\rangle$ and denote by $i^{\ast}_1$ the left adjoint functor to the inclusion $i_1\colon \mathscr{D}_0\hookrightarrow\mathscr{D}$ and $i_2^{!}$  the right adjoint functor to the inclusion $i_2\colon\mathscr{D}_2\hookrightarrow\mathscr{D}$. We consider the space of stability conditions on $\mathscr{D}_i$ with respect to $(\Lambda_i,v_i)$. By identifying $K(\mathscr{D})\cong K(\mathscr{D}_1)\oplus K(\mathscr{D}_2)$ and defining $\Lambda=\Lambda_1\oplus\Lambda_2$, we get a homomorphism $v:=v_1\oplus v_2\colon K(D)\to\Lambda$. The space of stability conditions on $\Stab(\mathscr{D})$ is then considered with respect to $(\Lambda,v)$. Throughout this paper, we consistently assume the existence of stability conditions on the semiorthogonal components of $\mathscr{D}$.

\begin{defnsub}
A stability condition $\sigma=(Z,\mathscr{P})$ on $\mathscr{D}$ is called \emph{glued from} $\sigma_1=(Z_1,\mathscr{P}_1)\in\Stab(\mathscr{D}_1)$ and $\sigma_2=(Z,\mathscr{P}_2)\in\Stab(\mathscr{D}_2)$ if 
\begin{enumerate}[$(1)$]
\item $Z_i=Z\vert_{\Lambda_i}$, for $i=1,2$;
\item $\mathscr{P}_i(\phi)\subset\mathscr{P}(\phi)$, for any $\phi\in\mathbb R$ and $i=1,2$.
\end{enumerate} 
\end{defnsub}

\begin{lemmasub}
\label{expressheart}
Let $\sigma=(Z,\mathscr{P})$ be a stability condition glued from $(\sigma_1=(Z_1,\mathscr{P}_1),\sigma_2=(Z_2,\mathscr{P}_2))\in\Stab(\mathscr{D}_1)\times\Stab(\mathscr{D}_2)$. Then we have $Z=Z_1\circ i^{\ast}_1+Z_2\circ i^{!}_2$
and
\[\mathscr{P}((0,1])=\left\{X\in \mathscr{D}\mid i^{\ast}_1(X)\in\mathscr{P}_1((0,1]), i^{!}_2(X)\in\mathscr{P}_2((0,1])\right\}.
\]
\end{lemmasub}
\begin{proof}
Note that every object $X\in\mathscr{D}$ fits into an exact sequence in $\mathscr{D}$
\[
0\to i^{!}_2(X)\to X\to i^{\ast}_1(X)\to0,
\]
which implies that $[E]=[i^{\ast}_1(E)]+[i^{!}_2(E)]$. Hence, we obtain $Z=Z_1\circ i^{\ast}_1+Z_2\circ i^{!}_2$. Furthermore, we deduce that $\langle\mathscr{P}_2((0,1),\mathscr{P}_1((0,1])\rangle$ forms a torsion pair in $\mathscr{P}((0,1])$. Therefore, 
\[
\left\{X\in \mathscr{D}\mid i^{\ast}_1(X)\in\mathscr{P}_1((0,1]), i^{!}_2(X)\in\mathscr{P}_2((0,1])\right\}=\langle\mathscr{P}_2((0,1),\mathscr{P}_1((0,1])\rangle\subset\mathscr{P}((0,1]).
\]
Since $\mathscr{P}_1((0,1])$ and $\mathscr{P}_2((0,1])$ are hearts of non-degenerate $t$-structures, we have $\mathscr{P}((0,1])=\langle\mathscr{P}_2((0,1),\mathscr{P}_1((0,1])\rangle$.
\end{proof}

Now we revisit the construction of gluing stability conditions on $\mathscr{D}$ derived from $\mathscr{D}_1$ and $\mathscr{D}_2$ in \cite{Collins2009GluingSC}, and a crucial ingredient of the construction is the notion of \emph{reasonable} stability conditions. 
\begin{defnsub}[\cite{Collins2009GluingSC}]
A stability condition $\sigma=(Z,H)\in\Stab(\mathscr{D})$ is called \emph{reasonable} if it satisfies
\[
\inf \left\{\vert Z(v(E))\vert\in\mathbb R\mid E\ne0 \text{ is semistable in } \sigma\right\}>0.
\]
\end{defnsub}

\begin{lemmasub}[{\cite{MR2376815,BayerMacri09,2021arXiv210713367K}}]
\label{equi-support}
Let $\sigma=(Z,\mathscr{P})$ be a pre-stability condition on a triangulated $\mathscr{D}$. The following conditions are equivalent:
\begin{enumerate}[$(1)$]
\item $\sigma$ is full and reasonable;
\item $\sigma$ is full and locally finite;
\item $\sigma$ satisfies the support property, i.e. $\sigma\in\Stab_{\Lambda}(\mathscr{D})$.
\end{enumerate}
\end{lemmasub}
\begin{proof}
$(1)\Rightarrow(2)$: Consider $\mathscr{A}:=\mathscr{P}((\phi-\epsilon,\phi+\epsilon))$ for some $\epsilon\in(0,\frac{1}{2})$. Since $\sigma$ is full, we can deform $\sigma$ to a discrete stability condition $\sigma'=(Z'\mathscr{P}')$ such that $\mathscr{A}\subset\mathscr{P}'(((\phi-\epsilon',\phi+\epsilon')))$, for some $\epsilon'\in[\epsilon,\frac12)$. Hence we may assume that $\sigma$ is discrete. Define a function $h\colon\mathscr{A}\to\mathbb R$ by $E\mapsto\Re(\mathrm{exp}(-\sqrt{-1}\pi\phi)Z(v(E)))$. As $\sigma$ is reasonable, there exists a constant $c>0$ with $h(A)>c$ for $0\ne A\in\mathscr{A}$. Note that the function $h$ is additive for strict short exact sequences. For any $E\in\mathscr{A}$, the central charges of all subobjects and quotient objects of $E$ lie within a bounded region of $\mathbb C$, which implies that $\sigma$ is locally finite.

$(2)\Rightarrow(3)$: Note that fullness of $\sigma$ is equivalent to the existence of $C>0$ such that for any $W\in\Hom_{\mathbb Z}(\Lambda,\mathbb C)$, we have 
\[
\sup\left\{\frac{\vert W(v(E))\vert}{\vert Z(v(E))\vert}\colon 0\ne E \text{ is } \sigma\text{-semistable } \right\}\le C.
\]
Assume that $\sigma$ does not satisfy the support property. Then there is a sequence $\{E_n\}$ of $\sigma$-semistable objects such that $\vert Z(v(E_n))\vert<\frac{\Vert v(E_n)\Vert}{n}$. Thus, for any $W\in\Hom_{\mathbb Z}(\Lambda,\mathbb C)$, we have
\[
\sup\left\{\frac{\vert W(v(E))\vert}{\vert Z(v(E))\vert}\colon 0\ne E \text{ is } \sigma\text{-semistable } \right\}\ge\frac{\vert W(v(E_n))\vert}{\vert Z(v(E_n))\vert}>n\cdot\frac{\vert W(v(E_n))\vert}{\Vert v(E_n)\Vert}.
\]
We can choose a sequence $\{W_n\}\in\Hom_{\mathbb Z}(\Lambda,\mathbb C)$ such that $\frac{\vert W_n(v(E_n))\vert}{\Vert v(E_n)\Vert}=1$, which contradicts the fullness of $\sigma$. Therefore, we conclude that $\sigma$ satisfies the support property.

$(3)\Rightarrow(1)$: Suppose $\sigma$ satisfies the support property. Then there is a constant $C>0$ such that for any $W\in\Hom_{\mathbb Z}(\Lambda,\mathbb C)$, we have
\[
\sup\left\{\frac{\vert W(v(E))\vert}{\vert Z(v(E))\vert}\colon 0\ne E \text{ is } \sigma\text{-semistable } \right\}\le C\cdot\sup\left\{\frac{\vert W(v(E))\vert}{\Vert v(E)\Vert}\colon 0\ne E \text{ is } \sigma\text{-semistable } \right\}.
\]
Since $\Lambda$ is of finite rank, it follows that $\sigma$ is full. By choosing a norm $\Vert\cdot\Vert$ such that $\Vert v\Vert\ge1$ for all $v\in\Lambda$, we can ensure that $\sigma$ is reasonable.
\end{proof}

\begin{corsub}
Let $\Sigma$ be a connected component of the space of pre-stability conditions on $\mathscr{D}$ that includes a stability condition $\sigma\in\Stab(\mathscr{D})$. Then $\Sigma\subset\Stab(\mathscr{D})$.
\end{corsub}
\begin{proof}
According to Lemma \ref{equi-support} and \cite[Proposition 1.2]{Collins2009GluingSC}, every pre-stability condition in $\Sigma$ is both full and reasonable. The claim then follows from Lemma \ref{equi-support}.
\end{proof}

\begin{thmsub}[{\cite[Theorem 3.6]{Collins2009GluingSC}}]
\label{gluestab}
Let $\langle\mathscr{D}_1,\mathscr{D}_2\rangle$ be a semiorthogonal decomposition of a triangulated category $\mathscr{D}$. Let $\sigma_i=(Z_i,\mathscr{P}_i)$ be a reasonable stability condition on $\mathscr{D}_i$ for $i=1,2$. Suppose that $\sigma_1$ and $\sigma_2$ satisfy the following conditions
\begin{enumerate}[$(1)$]
\item $\Hom^{\le0}_{\mathscr{D}}(\mathscr{P}_1((0,1]),\mathscr{P}_2((0,1]))=0$;
\item There is a real number $a\in(0,1)$ such that $\Hom^{\le0}_{\mathscr{D}}(\mathscr{P}((a,a+1]),\mathscr{P}((a,a+1]))=0$.
\end{enumerate}
Then there exists a unique reasonable stability condition $\mathrm{gl}(\sigma_1,\sigma_2)$ on $\mathscr{D}$ glued from $\sigma_1$ and $\sigma_2$.
\end{thmsub}

%According to \cite[Proposition 2.3.1]{MR1453167}, we have a short exact sequence of triangulated categories with exact functors\colon
%\[
%0\to\mathscr{D}_0\to\mathscr{D}\to \mathscr{D}/\mathscr{D}_0\to0.
%\]

%As in \ref{sec: stab}, we fix a finite rank lattice $\Lambda$ with a surjective group homomorphism $v\colon K(\mathscr{D})\twoheadrightarrow\Lambda$ and a norm $\Vert\cdot\Vert$ on $\Lambda_{\mathbb R}$.

\begin{defnsub}
\label{def-ref}
A relative pre-stability condition $\sigma_{\r}=(Z,\mathscr{P}_1)$ on $\mathscr{D}$ with respect to a left admissible subcategory $\mathscr{D}_1$ is defined by a group homomorphism $Z\colon \Lambda\to\mathbb C$ and a slicing $\mathscr{P}_1$ of $\mathscr{D}_1$ that fulfills the following conditions:
\begin{enumerate}[$(1)$]
\item $(Z,\mathscr{P}_1)$ can be \emph{relatively extended} to a pre-stability condition ${\sigma}=(Z,{\mathscr{P}})$ on $\mathscr{D}$, i.e. $\mathscr{P}_1(\phi)\subset{\mathscr{P}}(\phi)$ for any $\phi\in\mathbb R$, and $(Z_2,\mathscr{P}_2):=(Z\vert_{\Lambda_2},\mathscr{P}\cap\mathscr{D}_2)\in\Stab(\mathscr{D}_2)$;

\item There exist $0<\epsilon<\delta<1$ such that for any heart $\mathscr{A}_i\in\mathscr{P}_i((\delta-\epsilon,1+\delta+\epsilon))$ of a bounded $t$-structure of $\mathscr{D}_i$,
\[
\langle\mathscr{A}_2,\mathscr{A}_1\rangle:=\{X\in\mathscr{D}\mid i_1^{\ast}(X)\in\mathscr{A}_1, i^{!}_{2}(X)\in\mathscr{A}_2\}
\]
constitutes a heart of a bounded $t$-structure in $\mathscr{D}$.
\end{enumerate}
\end{defnsub}

Note that the relatively extended pre-stability condition $(Z,\mathscr{P})$ of $(Z,\mathscr{P}_1)$ is not unique in general. By definition, if $0\ne E\in\mathscr{P}_1(\phi)$, then $Z(v(E))=m(E)\mathrm{exp}(\sqrt{-1}\pi\phi)$ for some $m(E)\in\mathbb R_{>0}$. The linear map $Z\colon \Lambda\to\mathbb C$ will be referred as the \emph{relative central charge} of $\sigma$ and $\mathscr{P}_1$ is termed the \emph{relative slicing}  of $\sigma$. Objects in $\mathscr{P}_1(\phi)$ are considered \emph{relatively} $\sigma$-\emph{semistable} in $\sigma$ of phase $\phi$, and the simple objects of $\mathscr{P}_1$ are said to be \emph{relatively} $\sigma$-\emph{stable}. Moreover, if $(Z,\mathscr{P})$ from Definition \ref{def-ref} satisfies the support property with respect to $(\Lambda,v)$ as described in Definition \ref{def: support}, we denote $(Z,\mathscr{P}_1)$ as a relative stability condition on $\mathscr{D}$ with respect to $\mathscr{D}_1$. Denote by $\Stab(\mathscr{D},\mathscr{D}_1)$ the set of relative stability conditions with respect to $(\Lambda,v)$. As a corollary, for $\sigma_{\r}\in\Stab(\mathscr{D},\mathscr{D}_1)$, there is a constant $C>0$ such that for all relatively $\sigma$-semistable objects $0\ne E\in\mathscr{D}$, we have 
\[
\Vert v(E)\Vert\le C\vert Z(v(E))\vert.
\]

Similar to \cite[Proposition 5.3]{MR2373143}, we have the following result:
\begin{propsub}
\label{equi}
To give a relative stability condition $\sigma_{\r}=(Z,\mathscr{P}_1)$ on $\mathscr{D}$ with respect to $\mathscr{D}_1$ is equivalent to providing the following:
\begin{enumerate}[$(1)$]
\item A heart $\mathscr{A}_1$ of a bounded $t$-structure on $\mathscr{D}_1$ and a group homomorphism $Z\colon\Lambda\to\mathbb C$ such that 
\begin{enumerate}
\item $Z\circ v$ defines a stability function on a heart $\mathscr{A}$, containing $\mathscr{A}_1$, of a bounded $t$-structure on $\mathscr{D}$  and possesses the Harder-Narasimhan property;
\item $\mathscr{A}_2:=\cap\mathscr{D}_2$ is a heart of a bounded $t$-structure on $\mathscr{D}_2$, and $Z\vert_{\mathscr{A}_2}$-semistable objects in $\mathscr{A}_2$ are also $Z$-semistable.
\end{enumerate}[$(1)$]
\item $0<\epsilon<\delta<1$ such that for any heart $\mathscr{A}_i\subset\mathscr{A}_{i}(\epsilon,\delta)$, where $\mathscr{A}_{i}(\epsilon,\delta)$ is the extension-closed subcategory of $\mathscr{D}_i$ generated by  
\[
\left\{E\in\mathscr{D}_i\colon
\begin{aligned} 
E \text{ is } Z\text{-semistable with } \phi(E)\in(\delta-\epsilon,1] \emph{ or }\\
 E[-1] \text{ is } Z\text{-semistable with } \phi(E)\in(0,\delta+\epsilon)
\end{aligned}
\right\},
\]
the subset 
\[
\langle\mathscr{A}_2,\mathscr{A}_1\rangle:=\{X\in\mathscr{D}\mid i_1^{\ast}(X)\in\mathscr{A}_1, i^{!}_{2}(X)\in\mathscr{A}_2\}
\]
constitutes a heart of a bounded $t$-structure in $\mathscr{D}$.
\end{enumerate}
\end{propsub}
\begin{proof}
Given a relative stability condition $\sigma_{\r}=(Z,\mathscr{P}_1)$  in $\Stab(\mathscr{D},\mathscr{D}_1)$, we can derive a bounded $t$-structure $\mathscr{P}_1(>0)$ of $\mathscr{D}_1$ with the heart $\mathscr{A}_1=\mathscr{P}_1((0,1])\subset\mathscr{D}_1$. By definition, $\sigma_{\r}$ can be relatively extended to a stability condition $\sigma=(Z,\mathscr{P})\in\Stab(\mathscr{D})$. According to \cite[Proposition 5.3]{MR2373143}, $Z\circ v$ is a stability function on $\mathscr{A}:=\mathscr{P}((0,1])$ that possesses the Harder-Narasimhan property. Condition (b) holds because $(Z_2,\mathscr{P}_2)\in\Stab(\mathscr{D}_2)$. Condition (3) is naturally satisfied by taking $\mathscr{A}_i(\epsilon,\delta):=\mathscr{P}_i((\delta-\epsilon,1+\delta+\epsilon))$.

For the converse, let $\mathscr{A}_1$ be the heart of a bounded $t$-structure on $\mathscr{D}_1$, and assume that $Z\circ v$ is a stability function on $\mathscr{A}$ with the Harder-Narasimhan property. For any $\phi\in(0,1]$, denote $\mathscr{P}(\phi)$ (resp. $\mathscr{P}_1(\phi)$) as the full additive subcategory of $\mathscr{D}$ (resp. $\mathscr{D}_1$) consisting of semistable objects of $\mathscr{A}$ (resp. $\mathscr{A}_1$) with phase $\phi$, including the zero object. Then, by \cite[Proposition 5.3]{MR2373143} again, we conclude that $(Z,\mathscr{P})$ is a stability condition on $\mathscr{D}$. Note that $\mathscr{P}_1$ forms a slicing of $\mathscr{D}$ and $\mathscr{P}_1(\phi)\subset\mathscr{P}(\phi)$ for all $\phi\in\mathbb R$. 
Therefore, we can assert that $\sigma_{\r}$ is a relative stability condition in $\Stab(\mathscr{D},\mathscr{D}_1)$. 
\end{proof}

\begin{propsub}
\label{redundant}
Suppose $\mathscr{D}$ is a smooth, proper, and idempotent complete triangulated category with a left admissible subcategory $\mathscr{D}_1$. Then the Condition (2) in \Cref{def-ref} can be omitted, i.e.
\[
\Stab(\mathscr{D},\mathscr{D}_1)\cong\left\{(Z,\mathscr{P}_1)\in\Hom_{\mathbb Z}(\Lambda,\mathbb C)\times\mathrm{Slice}(\mathscr{D}_1)\colon
\begin{aligned}
&\exists\,\text{ relatively extended } (Z,\mathscr{P})\\
&\text{in }\Stab(\mathscr{D})\\
\end{aligned}
\right\}.
\]
\end{propsub}
\begin{proof}
Let $(Z,\mathscr{P}_1)\in\Hom_{\mathbb Z}(\Lambda,\mathbb C)\times\mathrm{Slice}(\mathscr{D}_1)$ with a relatively extended stability condition $\sigma=(Z,\mathscr{P})\in\Stab(\mathscr{D})$. According to \cite[Proposition 3.2]{HalpernLeistner2023StabilityCA}, there exists an integer $n$, dependent only on $\cD_i$, such that 
\[
\Hom^{\le n}(\mathscr{P}_1((0,1]),\mathscr{P}_2((0,1]))=0,
\]
which is equivalent to 
\[
\Hom_{\mathscr{D}}(\mathscr{P}_1((-\infty,1+n],\mathscr{P}_2(n,+\infty))=0.
\]
Consider $\sigma'_2=(Z_2',\mathscr{P}_2'):=\sigma_2[2m]$ with $2m\ge n+1$. Due to \Cref{gluestab}, $(Z_1,\mathscr{P}_1)$ and $\sigma_2[2m]$ can be glued into a stability condition $\sigma'=(Z,\mathscr{P}')$. Furthermore, we can obtain a family of hearts $\{\mathscr{A}'_{\epsilon}\colon\epsilon>0\}$ glued from $\mathscr{P}_1((\epsilon,1+\epsilon])$ and $\mathscr{P}'_2((\epsilon,1+\epsilon])$ by Theorem \ref{gluestab}, which implies that $(Z,\mathscr{P}_1)\in\Stab(\mathscr{D},\mathscr{D}_1)$.
\end{proof}

\begin{remksub}
\label{eredundant}
By \cite{ORLOV201659,1942_7648,kontsevich2008stabilitystructuresmotivicdonaldsonthomas,
Bkstedt1993HomotopyLI}, the result of Proposition \ref{redundant} is applicable to the bounded derived categories of coherent sheaves on any smooth projective complex variety, as well as to the bounded derived category of modules over a finite-dimensional algebra with finite global dimension.
\end{remksub}

\subsection{Deformations of relative stability conditions}
\label{sec: deform}
In this section, we demonstrate the deformation property of spaces of relative stability conditions via the deformation of gluing stability conditions, which extends the result in Theorem \ref{deformation}. The proof relies on the insights from \cite{MR2373143, Bayer2007ATT, Collins2009GluingSC,1cd9cb3351f44a768b6be019582fb4e3}. The main result is stated as follows:
\begin{thmsub}
\label{rdeformation}
Let $\mathscr{D}$ be a triangulated category with a left admissible subcategory $\mathscr{D}_1\subset\mathscr{D}$. The space of relative stability conditions $\Stab(\mathscr{D},\mathscr{D}_1)$ has the structure of a complex manifold, and the map 
\[
\cZ_1\colon\Stab(\mathscr{D},\mathscr{D}_1)\to\Hom_{\mathbb Z}(\Lambda,\mathbb C)
\]
that sends a relative stability condition to its relative central charge is a local isomorphism.
\end{thmsub}
In other words, Theorem \ref{rdeformation} allows us to deform a relative stability condition $(Z,\mathscr{P}_1)$ (uniquely) by deforming $Z$.  

We now define the function
\[
\begin{aligned}
d_{\r}\colon\Stab(\mathscr{D},\mathscr{D}_1)\times\Stab(\mathscr{D},\mathscr{D}_1)&\to\mathbb R\\
(\sigma_1,\sigma_2)&\mapsto d_1(\sigma_{1},\sigma_{2})+d_2(\sigma_{1},\sigma_{2}),
\end{aligned}
\]
where $\sigma_{i}=(Z_{i},\mathscr{P}_i)$, for $i=1,2$,
\[
d_1(\sigma_1,\sigma_2):=
\sup\limits_{0\ne E\in\mathscr{D}_1}\left\{|\phi_{\sigma_2}^-(E)-\phi_{\sigma_1}^-|, |\phi^+_{\sigma_2}(E)-\phi^+_{\sigma_1}(E)|, \left|\mathrm{log}\frac{m_{\sigma_2}(E)}{m_{\sigma_1}(E)} \right| \right\},
\]
and
\[
d_2(\sigma_1,\sigma_2):=\sup\limits_{0\ne E\in\mathscr{D}_2}\left\{\frac{\vert Z_{1}(E)-Z_{2}(E)\vert}{\Vert v(E) \Vert}\right\}.
\]
\begin{lemmasub}
\label{local-inj}
The function $d_{\r}$ induces a generalised metric on $\Stab(\mathscr{D},\mathscr{D}_1)$ with the property that if $d_{\r}(\sigma_1,\sigma_2)<1$ and $Z_1=Z_2$, then $\sigma_1=\sigma_2$.
\end{lemmasub}
\begin{proof}
Suppose that $d_{\r}(\sigma_1,\sigma_2)=0$. Then, an object $E\in\mathscr{D}$ is relatively $\sigma_1$-semistable if and only if $E$ is relatively $\sigma_2$-semistable, and for any nonzero $E\in\mathscr{D}_1$, we have $m_{\sigma_1}(E)=m_{\sigma_2}(E)$. It follows that $Z_1\vert_{\Lambda_1}=Z_2\vert_{\Lambda_1}$ since $(Z_1\vert_{\Lambda_1},\mathscr{P}_1)$ and $(Z_2\vert_{\Lambda_1},\mathscr{P}_2)$ are stability conditions on $\mathscr{D}_1$ and $Z_1\vert_{\Lambda_1},Z_2\vert_{\Lambda_1}$ agree on semistable objects which span $K(\mathscr{D}_0)$. In addition, $Z_1(E)=Z_2(E)$ for any nonzero $E\in\mathscr{D}_2$. Hence, we conclude that $\sigma_1=\sigma_2$ and it follows that $d_{\r}$ is a metric on $\Stab(\mathscr{D},\mathscr{D}_1)$. By \cite[Lemma 6.4]{MR2373143}, if $d_{\r}(\sigma_1,\sigma_2)<1$, we have $(Z_1\vert_{\Lambda_1},\mathscr{P}_1)=(Z_2\vert_{\Lambda_1},\mathscr{P}_2)$, which implies $\sigma_1=\sigma_2$. 
\end{proof}

Note that the term generalised metric means a distance function $d\colon X\times X\to[0,\infty]$ on a set $X$ satisfying all the usual metric space axioms except that it need not be finite. Based on Lemma \ref{local-inj}, $d_{\r}$ defines a topology on $\Stab(\mathscr{D},\mathscr{D}_1)$ in the usual way and induces a metric space structure on each connected component of $\Stab(\mathscr{D},\mathscr{D}_1)$.

\begin{proof}[Proof of Theorem \ref{rdeformation}]
The local injectivity follows from Lemma \ref{local-inj}. We now proceed to demonstrate the local surjectivity. Consider $\sigma_{\r}=(Z,\mathscr{P}_1)\in\Stab(\mathscr{D},\mathscr{D}_1)$ with its relatively extended stability condition $\sigma=(Z,\mathscr{P})\in\Stab(\mathscr{D})$. Let $\sigma_1=(Z_1,\mathscr{P}_1):=(Z\vert_{\Lambda_1},\mathscr{P}_1)\in\Stab(\mathscr{D}_1)$. According to Theorem \ref{deformation}, there exists an $\epsilon_1\in(0,\epsilon)$, where $\epsilon$ is as in Definition \ref{def-ref} for $\sigma_{\r}$, such that for any $W_1\in\Hom_{\mathbb Z}(\Lambda_1,\mathbb C)$ with $\Vert W_1-Z_1\Vert<\epsilon_1$, there is a unique stability condition $\tau_1=(W_1,\mathscr{Q}_1)\in\Stab(\mathscr{D}_1)$ with $d(\sigma_1,\tau_1)<\epsilon_1$. Suppose $\tau_1=(W_1,\mathscr{Q}_1)$ is such a stability condition in $\Stab(\mathscr{D}_1)$ with $\Vert Z_1-W_1\Vert<\epsilon'$. We claim that there exists a stability condition $\tau_2=(W_2,\mathscr{Q}_2)\in\Stab(\mathscr{D}_2)$ that can be glued together with $\tau_1$. We can first reduce to the case where $\Im Z_1=\Im W_1$ or $\Re Z_1=\Re W_1$. Indeed, it suffices to consider the first case, as the second follows the same way by replacing $\mathscr{P}_1((0,1])$ with $\mathscr{P}_1((\frac{1}{2},\frac{3}{2}])$. In this scenario, we have $\mathscr{Q}_1((0,1])=\mathscr{P}_1((0,1])$. By \cite[Lemma 6.1]{MR2373143}, it follows that
\[
\mathscr{Q}_1(\phi)\subset\mathscr{P}_1([\phi-\epsilon_1,\phi+\epsilon_1]) \text{ for all } \phi\in\mathbb R.
\]
Thus, 
\[
\mathscr{Q}_1((\delta,1+\delta])\subset\mathscr{P}_1((\delta-\epsilon_1,1+\delta+\epsilon_1]).
\]
Let us denote $\mathscr{A}_{\alpha}:=\langle\mathscr{P}_2(\alpha,1+\alpha), \mathscr{Q}_1(\alpha,1+\alpha)\rangle$ for any $\alpha\in\mathbb R$. By Definition \ref{def-ref}, we obtain a heart $\mathscr{A}_{\delta}$ of a bounded $t$-structure in $\mathscr{D}$. Define
\[
\mathscr{P}'((0,\delta]):=\left\{X\in \mathscr{D}\mid i^{\ast}_1(X)\in \mathscr{Q}_1((0,\delta]), i^{!}_2(X)\in\mathscr{P}_2((0,\delta])\right\},
\]
and
\[
\mathscr{P}'((\delta,1]):=\left\{X\in \mathscr{D}\mid i^{\ast}_1(X)\in \mathscr{Q}_1((\delta,1]), i^{!}_2(X)\in\mathscr{P}_2((\delta,1])\right\}.
\]
We claim that $\langle\mathscr{P}'((\delta,1]),\mathscr{P}'((0,\delta])\rangle$ forms a torsion pair in $\mathscr{A}_0=\mathscr{P}((0,1])$. For any $X\in\mathscr{A}_0$, we need to demonstrate the existence of a unique short exact sequence
\[
0\to X_2\to X\to X_1\to0
\]
with $X_2\in\mathscr{P}'((\delta,1])$ and $X_1\in\mathscr{P}'((0,\delta])$. Note that $\mathscr{A}_0\subset\langle\mathscr{A}_{\delta},\mathscr{A}_{\delta}[-1]\rangle$ and $\mathscr{A}_{\delta}\subset\langle\mathscr{A}_{0}[1],\mathscr{A}_{0}\rangle$. Hence, we have a unique triangle
\begin{equation}
\label{torsion-sequence}
X_2\to X\to X_1\to X_2[1].
\end{equation}
It is straightforward to verify that $X_1, X_2\in\mathscr{A}_0$, which implies Equation (\ref{torsion-sequence}) is the required sequence. We now define 
\[
Z':=W_1\circ i^{\ast}_1+Z_2\circ i^{!}_2\in\Hom_{\mathbb Z}(\Lambda,\mathbb C).
\]
Then $Z'$ is a stability function on $\mathscr{A}_0$. Observe that for $X\in\mathscr{P}'((0,\delta])$, we have $\phi(Z'(X))\le\delta$, and for $X\in\mathscr{P}'((a,1])$, $\phi(Z'(X))>\delta$. Let $0\ne E\in\mathscr{P}'((\delta,1])$. Assume $X\in\mathscr{A}_0$ destabilizes $E$, which means $X\subset E$ with $\phi(Z'(E))<\phi(Z'(X))$. Consider the exact sequence
\[
0\to X_2\to X\to X_1\to 0,
\]
where $X_2\in\mathscr{P}'((\delta,1])$ and $X_1\in\mathscr{P}'((0,\delta])$. We have that $\phi(Z'(X_2))\ge\phi(Z'(X))>\phi(Z'(E))$, which implies $X_2\in\mathscr{P}'((\delta,1])$ also destabilizes $E$. According to Lemma \ref{equi-support}, $\mathscr{Q}_1((0,\delta]), \mathscr{Q}_1((\delta,1]),\mathscr{P}_2((0,\delta]),\mathscr{P}((\delta,1])$ are all finite length. Therefore, the Harder-Narasimhan filtration of $X\in\mathscr{P}'((0,1])$ for $(Z',\mathscr{P}')$ can be induced from the Harder-Narasimhan filtrations of $X_2$ and $X_1$ from the canonical exact sequence of $X$. Thus, $(Z',\mathscr{P}')$ constitutes a stability condition on $\mathscr{D}$ which completes the proof of the claim by taking $\tau_2=\sigma_2$. Next, we demonstrate that $\sigma'_{\r}:=(Z',\mathscr{Q}_1)$ is a relative stability condition in $\Stab(\mathscr{D},\mathscr{D}_1)$. It suffices to verify condition (2) in Definition \ref{def-ref} which can be obtained by choosing $\epsilon'\in(0,\epsilon-\epsilon_1)$ and $\delta'=\delta$. We can then conclude that for any $\Vert W-Z\Vert<\epsilon_1$, there is a unique relative stability condition $\tau_{\r}=(W,\mathscr{Q}_1)\in\Stab(\mathscr{D},\mathscr{D}_1)$ such that $d_{\r}(\sigma_{\r},\sigma'_{\r})<\epsilon_1$.
\end{proof}

It is worth noting that there are also two types of group actions on the spaces of relative stability conditions. The action of $\gl2$ is defined in a manner analogous to that for stability conditions. One can verify that the spaces of relative stability conditions are closed under the $\gl2$-action, ensuring that this action is well-defined.  As for the action of the automorphism group, we need to confine it to its subgroup $\mathrm{Aut}(\mathscr{D},\mathscr{D}_1):=\{T\in\mathrm{Aut}(\mathscr{D})\colon T(\mathscr{D}_1)\subset\mathscr{D}_1\}$. 

\begin{remksub}
\label{right-convention}
In this paper, we investigate relative stability conditions on triangulated categories with respect to left admissible subcategories. By analogy, one can also define relative stability conditions with respect to right admissible subcategories and explore the relations between these relative stability conditions in certain cases.

According to \cite{MR2524593,qiu2023fusionstablestructurestriangulationcategories,dell2024fusionequivariantstabilityconditionsmorita}, one can also  consider the \emph{fusion-equivariant relative stability conditions}. Specifically, given a fusion category, one can examine the space $\Stab(\mathscr{D},\mathscr{D}_1)^{\cC}$ of $\cC$-equivariant relative stability conditions. It is then straightforward to extend \cite[Theorem A]{dell2024fusionequivariantstabilityconditionsmorita} to the relative setting, i.e. the space $\Stab(\mathscr{D},\mathscr{D}_1)^{\cC}$ is a closed submanifold of $\Stab(\mathscr{D},\mathscr{D}_1)$ with the forgetful map $Z$ defining the local homeomorphism 
\[
Z\colon\mathrm{Stab}(\mathscr{D},\mathscr{D}_1)^{\cC}\to\Hom_{K(\cC)}(K(\mathscr{D}),\mathbb C),\quad (Z,\mathscr{P})\mapsto Z.
\]
Suppose a finite group $G$ acts on $\mathscr{D}$ which preserves $\mathscr{D}_1$. Let $\mathscr{D}^G$ and $\mathscr{D}_1^G$ denote the $G$-equivariantisations of $\mathscr{D}$ and $\mathscr{D}_1$, respectively. Let $\mathrm{vec}(G)$ and $\mathrm{rep}(G)$ denote the fusion categories of finite dimensional $G$-graded vector spaces and representations of $G$, respectively. Under certain assumptions on $(\mathscr{D},\mathscr{D}_1)$, \cite[Theorem 4.8]{dell2024fusionequivariantstabilityconditionsmorita} can also be generalized to the relative context. Namely, the space of $\mathrm{rep}(G)$-equivariant relative stability conditions on $\mathscr{D}^G$ with respect to $\mathscr{D}_1^G$ is homeomorphic to the space of $\mathrm{vec}(G)$-equivariant relative stability conditions on $\mathscr{D}$ with repsect to $\mathscr{D}_1$. 
\end{remksub}

Let $X$ be a smooth projective complex variety and let $\Stab(X)$ denote the space of numerical stability conditions on the bounded derived category of coherent sheaves on $X$, denoted by $\mathscr{D}^b(X)$. Recall that a stability condition $\sigma\in\Stab(X)$ is referred to as \emph{geometric} if, for every point $x\in X$, the skyscraper sheaf $\cO_x$ is $\sigma$-stable. Consider $\mathscr{D}_1$ a left admissible subcategory of $\mathscr{D}^b(X)$, and let $\Stab(X,\mathscr{D}_1)$ represent the space of numerical relative stability conditions on $\mathscr{D}^b(X)$ with respect to $\mathscr{D}_1$. A relative stability condition $\sigma_{\r}\in\Stab(X,\mathscr{D}_1)$ is referred to as \emph{geometric} if there is stability condition $\sigma\in\Stab(X)$ that relatively extends $\sigma_{\r}$ and is geometric. For surfaces, we have a further description for geometric relative stability conditions as follows:
\begin{lemmasub}
\label{rel-geo}
Let $X$ be a smooth projective complex surface and $\mathscr{D}_1$ a left admissible subcategory of $\mathscr{D}^b(X)$. Any geometric relative stability condition $\sigma_{\r}=(Z,\mathscr{P}_1)\in\Stab(X,\mathscr{D}_1)$ is uniquely determined by its relative central charge up to shifting the slicing by $[2n]$ for any $n\in\mathbb Z$. Furthermore, if $\sigma_{\r}$ is normalised by the action of $\gl2$ such that $Z(\cO_x)=-1$ and $\phi(\cO_x)=1$ for any $x\in X$, then:
\begin{enumerate}[$(1)$]
\item the relative central charge can be uniquely written in the following form:
\[
Z(E)=(\alpha-i\beta)H^2\mathrm{ch}_0(E)+(B+iH)\cdot\mathrm{ch}_1(E)-\mathrm{ch}_2(E),
\]
where $\alpha,\beta\in\mathbb R$, and $(H,B)\in\mathrm{Amp}_{\mathbb R}(X)\times\mathrm{NS}_{\mathbb R}(X)$.
\item $\mathscr{P}((0,1])$ is given by the tilt of $\mathrm{Coh}(X)$ at the torsion pair $(T,F)$, where
\[
T:=\left\{E\in\mathrm{Coh}(X)\colon
\begin{aligned}
&\text{ Any } H\text{-semistable Harder-Narasimhan factor } F \text{ of the }\\
&\text{ torsion free part of } E \text{ satisfies }\mathrm{Im} Z([F])>0.
\end{aligned}
\right\},
\]
\[
F:=\left\{E\in\mathrm{Coh}(X)\colon
\begin{aligned}
&E \text{ is torsion free, and any } H\text{-semistable Harder-Narasimhan}\\
&\text{ factor } F \text{ of } E \text{ satisfies }\mathrm{Im} Z([F])\le0.
\end{aligned}
\right\}.
\]
\item $\mathscr{P}_1((0,1])$ is given by $\mathscr{P}((0,1])\cap\mathscr{D}_1$.
\end{enumerate}
\end{lemmasub}
\begin{proof}
The result follows directly from {\cite[Proposition 10.3]{MR2376815}}.
\end{proof}

%================================================================
%================================================================

\subsection{Initial examples}
\label{initial-example}
In this section, we present some initial examples of spaces of relative stability conditions on triangulated categories arising from Dynkin quivers and curves with positive genus. We start with some trivial examples:
\begin{enumerate}[$(1)$]
\item If $\mathscr{D}_1=0$, $\Stab(\mathscr{D},\mathscr{D}_1)$ consists of all the central charges in $\Stab(\mathscr{D})$;
\item If $\mathscr{D}_1=\mathscr{D}$, $\Stab(\mathscr{D},\mathscr{D}_1)=\Stab(\mathscr{D})$;
\item If $\mathscr{D}=\langle\mathscr{D}_1,\mathscr{D}_2\rangle$ forms an orthogonal decomposition, then
\[\Stab(\mathscr{D},\mathscr{D}_1)\cong\Stab(\mathscr{D}_1)\times\Stab(\mathscr{D}_2,0).\]
\end{enumerate}

\subsubsection{Categories from Dynkin quivers}
Let $Q$ be a Dynkin quiver and denote by $\mathscr{D}^b(Q)$ the bounded derived category of finitely generated right $\mathbb CQ$-modules. Consider a left admissible subcategory $\mathscr{D}_1\subset\mathscr{D}^b(Q)$. According to \cite{MR2575093} and \cite{MR2366948}, there exists a bijection between the following classes:
\begin{enumerate}[$(1)$]
\item The class of thick subcategories of $\mathscr{D}^b(Q)$;
\item The class of abelian subcategories of $\mathrm{mod}$-$\mathbb CQ$ which are closed under extensions;
\item Finitely generated torsion classes in $\mathrm{mod}$-$\mathbb CQ$.
\end{enumerate}
Consequently, $\mathscr{D}_1$ can also be characterized by the aforementioned classes. Let $\mathscr{A}$ be a heart in $\mathscr{D}^b(Q)$ and $S\in\mathscr{A}$ a simple object. We now review the concept of \emph{simple tilts} in \cite{MR3406522}. Define the full subcategories
\[
^{\perp}S:=\{E\in\mathscr{A}\mid\Hom_{\mathscr{A}}(E,S)=0\}, \quad S^{\perp}:=\{E\in\mathscr{A}\mid\Hom_{\mathscr{A}}(S,E)=0\}.
\]
The \emph{forward simple tilt} of $\mathscr{A}$ by $S$ is then given by $\mathscr{A}^{\#}_S:=\langle S[1],^{\perp}S\rangle$ and the \emph{backward simple tilt} of $\mathscr{A}$ by $S$ is given by $\mathscr{A}_S^{\flat}:=\langle S^{\perp},S[-1]\rangle$. It is known that the subcategories $\mathscr{A}^{\#}_S$ and $\mathscr{A}_S^{\flat}$ are hearts in $\mathscr{D}^b(Q)$.  The \emph{exchange graph} $\mathrm{EG}(\mathscr{D})$ \cite{MR3406522} of a triangulated category $\mathscr{D}$ is defined to be the oriented graph whose vertices are hearts in $\mathscr{D}$ and whose edges represent forward simple tilts between these hearts.  We can then describe $\Stab(\mathscr{D}^b(Q),\mathscr{D}_1)$ in terms of the set of vertices on $\mathrm{EG}(\mathscr{D})$, denoted by $\mathrm{EG}_0(\mathscr{D}_1)$, as follows:
\begin{propsub}
\label{ADE}
Suppose $\mathscr{D}^b(Q)$ denotes the bounded derived category of finitely generated right $\mathbb CQ$-modules for a Dynkin quiver $Q$, with $\mathscr{D}_1\ne0$ being a left admissible subcategory of $\mathscr{D}^b(Q)$. Then we have
\[
\Stab(\mathscr{D}^b(Q),\mathscr{D}_1)\cong\bigcup_{\mathscr{A}\in\mathrm{EG}_0(\mathscr{D}_1)}\mathbb H_+^{n},
\]
where $n=\mathrm{rk}(K(\mathscr{D}^b(Q)))$. 
\end{propsub}
\begin{proof}
Let $\sigma_{\r}=(Z,\mathscr{P}_1)\in\Stab(\mathscr{D}^b(Q),\mathscr{D}_1)$ and denote by $\mathscr{A}_1:=\mathscr{P}_1((0,1])$. We define the subset $U(\mathscr{A}_1)\subset\Stab(\mathscr{D}^b(Q),\mathscr{D}_1)$ as 
\[
{U}(\mathscr{A}_1):=\{(Z,\mathscr{P}_1)\in\Stab(\mathscr{D}^b(Q),\mathscr{D}_1)\mid\mathscr{P}_1((0,1])=\mathscr{A}_1\}.
\]
Denote by
\[
\cS(\mathscr{D},\mathscr{D}_1):=\left\{(Z,\mathscr{P}_1)\in\Hom_{\mathbb Z}(\Lambda,\mathbb C)\times\mathrm{Slice}(\mathscr{D}_1)\colon
\begin{aligned}
&\exists\,\sigma=(Z,\mathscr{P})\in\Stab(\mathscr{D}^b(Q)) \text{ such that } \\
&\mathscr{P}_1(\phi)\subset\mathscr{P}(\phi) \text{ for any } \phi\in\mathbb R.
\end{aligned}
\right\}.
\]
We claim that
\[
\Stab({\mathscr{D}^b(Q),\mathscr{D}_1})\cong \cS(\mathscr{D}^b(Q),\mathscr{D}_1).
\]
It is clear that $\Stab({\mathscr{D}^b(Q),\mathscr{D}_1})\subset\cS(\mathscr{D}^b(Q),\mathscr{D}_1)$. Consider $(Z,\mathscr{P}_1)\in\cS(\mathscr{D},\mathscr{D}_1)$ with $(Z,\mathscr{P})\in\Stab(\mathscr{D})$. Based on the discussion above, we have $\langle\mathscr{P}_2:=\mathscr{P}\cap\mathscr{D}_2,\mathscr{P}_1\rangle$ forms a torsion pair. Since $\mathscr{P}_2((0,1])$ is of finite length and has finite number of simple objects, we can apply the strategy from the proof of \Cref{deformation} to conclude that every $Z_2:=Z\vert_{\mathscr{D}_2}$-semistable object is also $Z$-semistable. This implies that $(Z,\mathscr{P}_1)$ can be relatively extended to $(Z,\mathscr{P})\in\Stab(\mathscr{D})$. Condition (2) in \Cref{def-ref} follows trivially due to the finiteness of $\mathscr{P}_i((0,1])$. Therefore, $\Stab({\mathscr{D}^b(Q),\mathscr{D}_1})\subset\cS(\mathscr{D},\mathscr{D}_1)$. According to \cite[Lemma 5.2]{Bridgeland2006SpacesOS}, we then have $U(\mathscr{A}_1)\cong\mathbb H_+^{n}$. Furthermore, based on \cite{MR976638}, any heart of a bounded $t$-structure on $\mathscr{D}^b(Q)$ can be obtained from the standard heart by iterated simple tilts, which leads to the conclusion that
\[
\Stab(\mathscr{D}^b(Q),\mathscr{D}_1)\cong\bigcup_{\mathscr{A}\in\mathrm{EG}_0(\mathscr{D}_1)}\mathbb H_+^{n}.
\]
\end{proof}

\begin{remksub}
Note that Proposition \ref{ADE} implies that the homotopy type of $\Stab(\mathscr{D}^b(Q),\mathscr{D}_1)$ is equivalent to that of $\Stab(\mathscr{D}_1)$. Therefore, if $\mathscr{D}_1$ is of Dynkin type, then $\Stab(\mathscr{D},\mathscr{D}_1)$ is simply connected due to \cite[Theorem 4.7]{MR3281136}. In general, one can also examine spaces of relative stability conditions on other categories related to these quivers.
\end{remksub}
%================================================================

\subsubsection{Categories from curves with positive genus}
Let $C$ be a smooth projective complex curve of genus $g\ge1$. The nonexistence of non-trivial semiorthogonal decompositions of $\mathscr{D}^b(C)$ has been established in \cite[Theorem 1.1]{OKAWA20112869}. As an application of the relative stability, we now provide an alternative proof for this result:
\begin{propsub}[{\cite[Theorem 1.1]{OKAWA20112869}}]
\label{relative-curve}
The bounded derived categories of coherent sheaves on smooth projective curves with positive genus does not admit any non-trivial semiorthogonal decomposition.
\end{propsub}
\begin{proof}
Suppose $\mathscr{D}_1\ne0$ is a proper left admissible subcategory of $\mathscr{D}^b(C)$. Note that all the line bundles and skyscraper sheaves are stable for any stability condition in $\mathscr{D}^b(C)$. Hence, stability conditions on $\mathscr{D}^b(C)$ can be restricted to $\mathscr{D}_1$ and $^{\perp}\mathscr{D}_1$, which implies the existence of stability conditions on $\mathscr{D}_1$ and $^{\perp}\mathscr{D}_1$. By Proposition \ref{redundant} and Remark \ref{eredundant}, there exist relative stability conditions in $\Stab(\mathscr{D}^b(C),\mathscr{D}_1)$. According to \cite{MR2335991}, the action of $\gl2$ on $\Stab(\mathscr{D}^b(C),\mathscr{D}_1)$ is both free and transitive, which leads to the following isomorphism 
\[
 \begin{aligned}
\Psi\colon\Stab(\mathscr{D}^b(C))&\to\Stab(\mathscr{D}^b(C),\mathscr{D}_1)\\
(Z,\mathscr{P})&\mapsto(Z,\mathscr{P}\cap\mathscr{D}_1)
 \end{aligned}   
\]
However, by Proposition \ref{redundant} and Remark \ref{eredundant}, this would lead to a contradiction since the selection of relatively extended stability conditions is not unique. By changing the convention as discussed in \Cref{right-convention}, we conclude that $\mathscr{D}^b(Q)$ is the only non-zero right and left admissible subcategory of $\mathscr{D}^b(Q)$.
\end{proof}

\section{Relative stability and the existence of the dHYM metric on line bundles}
\label{sec: relstab-dHYM}
In this section, we first provide a brief overview of the background on dHYM equations. Then we propose a question concerning the relations between the relative stability and solvability of dHYM equations. Afterward, we examine the question through some concrete cases.

\subsection{Reviews of dHYM equations}
Let $X$ be a smooth projective complex variety and $\omega$ be an ample $\R$-divisor on $X$. 
\begin{defnsub}
Let $\alpha$ be a real $(1,1)$-form on $X$. The \emph{deformed Hermitian-Yang-Mills (dHYM) equation} seeks a function $\phi\colon X\to\R$ such that $\alpha_{\phi}=\alpha+\sqrt{-1}\partial\overline{\partial}\phi$, which satisfies
$$
\Im(e^{-\sqrt{-1}\widehat{\theta}}(\omega+\sqrt{-1}\alpha_{\phi})^n)=0,
$$
where
$$
\int_X(\omega+\sqrt{-1}\alpha_{\phi})^n\in \R_{>0}e^{\sqrt{-1}\widehat{\theta}}.
$$
\end{defnsub}
\begin{remksub}
If we fix a point $p\in X$ and choose a holomorphic coordinate $\{z^i\}$ centered at $p$ such that
$$
\omega=\sqrt{-1}\sum_i\mathrm{d}z^i\wedge\mathrm{d}\overline{z}^i,\quad \alpha_{\phi}=\sqrt{-1}\sum_i\lambda_i\mathrm{d}z^i\wedge\mathrm{d}\overline{z}^i,
$$
then the dHYM equation can be written as
$$
\Theta_{\omega}(\alpha_{\phi})=\widehat{\theta}\pmod{2\pi},
$$
where $\Theta_{\omega}(\alpha_{\phi})=\sum_i\arctan(\lambda_i)$ is called the \emph{Lagrangian phase operator}. 
\end{remksub}
\begin{defnsub}
Let $\mathcal{L}\to(X,\omega)$ be a line bundle. A Hermitian metric $h$ on $\mathcal{L}$ is called a dHYM metric with respect to $\omega$ if the Chern curvature $\Theta_h$ satisfies
$$
\Im\left(e^{-\sqrt{-1}\widehat{\theta}}\left(\omega-\frac{\Theta_h}{2\pi}\right)^n\right)=0,
$$
where
$$
\int_X\left(\omega-\frac{\Theta_h}{2\pi}\right)^n\in\R_{>0}e^{\sqrt{-1}\widehat{\theta}}.
$$
\end{defnsub}
\begin{remksub}
\noindent
\begin{enumerate}[$(1)$]
\item The dHYM metric on a line bundl $\mathcal{L}$ is a special case of the dHYM equation. If we choose real $(1,1)$-form $\alpha=\ch_1(\mathcal{L})$, then a solution of dHYM equation gives a dHYM metric on $\mathcal{L}$. 
\item Given a $\R$-divisor $B$ on $X$, which is called a \emph{$B$-field} in literature, a dHYM metric with respect to $\omega$ and $B$ is a solution of dHYM equation defined by real $(1,1)$-class $\ch_1^B(\mathcal{L})$, where $\ch^B_1(\mathcal{L})=e^{-B}\ch_1(\mathcal{L})$ is the \emph{twisted Chern character}.
\item The higher rank version of dHYM equation was proposed by Collins-Yau in \cite[\&8.1]{collins2018moment}. For a holomorphic vector bundle $\mathcal{E}\to (X,\omega)$, a Hermitian metric $h$ is called a dHYM metric if the Chern curvature $\Theta_h$ satisfies
$$
\Im\left(e^{-\sqrt{-1}\widehat{\theta}}\left(\omega\otimes\id_{\mathcal{E}}-\frac{\Theta_h}{2\pi}\right)^n\right)=0,
$$
where
$$
\int_X\tr_h\left(\omega\otimes\id_{\mathcal{E}}-\frac{\Theta_h}{2\pi}\right)^n\in\R_{>0}e^{\sqrt{-1}\widehat{\theta}}
$$
and the imaginary part is defined using the metric $h$. There are many fundamental results about dHYM metric for the line bundle, such as \cite{MR3694663,collins2018moment,MR4091029}, but the existence of the solution to the higher rank version is still in mystery.
\end{enumerate}
\end{remksub}

In \cite{MR4091029}, the authors establish a Nakai-Moishezon type criterion for the existence of solutions to the dHYM equation on K\"ahler surfaces, which is referred to as the twisted ampleness criterion in \cite{collins2023stabilitylinebundlesdeformed}. The criterion for the higher dimensions is proved in \cite{MR4728704}. Specifically, the twisted ampleness criterion for K\"ahler surfaces is formulated as follows:
\begin{thmsub}
\label{ample-criterion}  
Let $(X,\omega)$ be a K\"ahler surface and $\mathcal{L}$ be a line bundle on $X$ with $\omega\cdot\ch_1(\mathcal{L})>0$. Then $\mathcal{L}$ admits a dHYM metric if and only if for every curve $C\subseteq X$, the following inequality holds:
$$
\Im\left(\frac{Z_C(\mathcal{L})}{Z_X(\mathcal{L})}\right)>0,
$$
where $Z_C(\mathcal{L})=-\int_Ce^{-\sqrt{-1}\omega}\ch(\mathcal{L})$ and $Z_X(\mathcal{L})=-\int_Xe^{-\sqrt{-1}\omega}\ch(\mathcal{L})$.
\end{thmsub}
\begin{remksub}
By results in \cite{fu2023deformedhermitianyangmillsflow}, the solution to the dHYM equation on a compact K\"ahler surface always exists on the complement of a finite number of curves of negative self-intersection. As a consequence, the twisted ampleness criterion is inherently satisfied if there is no negative self-intersection curve on $X$. Furthermore, it's also interesting to consider whether it suffices to test negative self-intersection curve in the twisted ampleness criterion or not, and it's true for some cases (Lemma \ref{lemma: twisted ampleness criterion on Hirzebruch surface}).
\end{remksub}

\subsection{Relations between relative stability and the dHYM equation}
%The relation between slope stability and the existence of Hermitian-Einstein metrics has been explored extensively in studies such as \cite{MR0184252,MR0710055,MR0765366,MR0861491,MR0944577}. The large volume limit of Bridgeland stability is understood to be linked to slope stability. Consequently, it is intriguing to examine the relations between Bridgeland stability and the existence of dHYM on line bundles. The work in \cite{MR4479717} demonstrates that on $\Bl_p\P^2$, every line bundle that admits a dHYM metric with for a given ample $\R$-divisor $\omega$ is Bridgeland stable with respect to a stability condition associated with $\omega$. However, the reverse implication does not hold. This result was later extended in \cite{collins2023stabilitylinebundlesdeformed} for Weierstrass elliptic K3 surfaces. In a broader context, \cite[Question 3.3]{collins2023stabilitylinebundlesdeformed} puts forth a conjecture suggesting that on any smooth projective complex surface, twisted ampleness is a sufficient condition for Bridgeland stable. In order to establish the converse, we aim to contract the hearts of stability conditions, which leads to the introduction of the concept of relative stability.

In this section, we focus on numerical relative stability conditions, where $\Lambda$ represents the numerical Grothendieck group. Consider $X$ be a smooth projective complex surface with ample $\R$-divisor $\omega$ and $B$ be a $\R$-divisor on $X$. Let $\sigma_{\omega,B}=(Z_{\omega,B},\mathscr{A}_{\omega,B})$ denote the Bridgeland stability with the central charge $Z_{\omega,B}$ given by
$$
Z_{\omega,B}=-\int_Xe^{-\sqrt{-1}\omega}\ch^B
$$ 
and $\mathscr{A}_{\omega,B}=\langle\mathscr{F}_{\omega,B}[1],\mathscr{T}_{\omega,B}\rangle$ where
$$
\begin{aligned}
\mathscr{T}_{\omega,B}&=\{\mathcal{E}\in \coh(X)\mid \mu_{\omega,B,\min}(\mathcal{E})>0\}\\
\mathscr{F}_{\omega,B}&=\{\mathcal{E}\in \coh(X)\mid \mu_{\omega,B,\max}(\mathcal{E})\leq0\}
\end{aligned}
$$
where $\mu_{\omega,B}(\mathcal{E})=\ch_1^B(\mathcal{E})/\ch_0(\mathcal{E})$. It is known that $\sigma_{\omega,B}$ is a geometric Bridgeland stability. Furthermore, in \cite[Theorem 5.10]{dell2023stabilityconditionsfreeabelian}, the author has classified all geometric Bridgeland stabilities for any smooth projective complex surface. 

The relation between Bridgeland stability and the existence of dHYM on line bundles has been a topic of extensive study. The work in \cite{MR4479717} shows that on $\Bl_p\P^2$, every line bundle that admits a dHYM metric with for a given ample $\R$-divisor $\omega$ is Bridgeland stable with respect to $\sigma_{\omega,0}$. However, the converse is not necessarily true. We now present this counter-example to the equivalence of Bridgeland stability and the existence of dHYM on line bundles for smooth projective complex surfaces as follows, with reference to \cite{MR4479717}:

\begin{propsub}\label{prop: surf-with-one}
Let $X$ be a smooth projective complex surface with only one negative self-intersection curve and $\omega$ be an ample $\R$-divisor on $X$. Consider a line bundle $\mathcal{L}$ on $X$ with $\omega\cdot\ch_1(\mathcal{L})>0$. If $\mathcal{L}$ admits a dHYM metric, then $\mathcal{L}$ is $\sigma_{\omega,0}$-stable. However, the converse is not valid.
\end{propsub}
\begin{proof}
A line bundle $\mathcal{L}\in\mathscr{A}_{\omega,0}$ if and only if $\omega\cdot\ch_1(\mathcal{L})>0$, and it is $\sigma_{\omega,0}$-stable if and only if 
\begin{equation}\label{eq: example in collins-Shi, Bridgeland part0}
\rho(\mathcal{L}(-C))<\rho(\mathcal{L}),
\end{equation}
where $\rho$ denotes the slope function associated with $\sigma_{\omega,0}$. A direct computation shows that Equation (\ref{eq: example in collins-Shi, Bridgeland part0}) is equivalent to 
\begin{equation}\label{eq: example in collins-Shi, Bridgeland part}
\left(C\cdot\ch_1(\mathcal{L})-\frac12C^2\right)\left(\omega\cdot\ch_1(\mathcal{L})\right)>\left(\ch_2(\mathcal{L})-\frac12\omega^2\right)\left(C\cdot\omega\right),
\end{equation}
where $C\subseteq X$ is the unique curve with negative self-intersection.

By the twisted ampleness criterion in Theorem \ref{ample-criterion}, $\mathcal{L}$ admits a dHYM metric if and only if
\begin{equation}\label{eq: example in collins-Shi, dHYM part}
\left(C\cdot\ch_1(\mathcal{L})\right)\left(\omega\cdot\ch_1(\mathcal{L})\right)>\left(\ch_2(\mathcal{L})-\frac12\omega^2\right)\left(C\cdot\omega\right),
\end{equation}
for all curves $C\subseteq X$. As a consequence, Equation (\ref{eq: example in collins-Shi, dHYM part}) implies Equation (\ref{eq: example in collins-Shi, Bridgeland part}), meaning that any line bundle admitting a dHYM metric is $\sigma_{\omega,0}$-stable.

For the converse statement, we consider $\omega=\frac{1}{\sqrt{3}}(D_1+D_4)$ and $\ch_1(\mathcal{L})=2D_4$. Then $\mathcal{L}$ is a line bundle that is Bridgeland stable with respect to $\sigma_{\omega,0}$ but does not admit a dHYM metric by above numerical criterions.
\end{proof}

Based on Proposition \ref{prop: surf-with-one}, in order to establish the converse, one straightforward approach is to contract the hearts of stability conditions. This strategy leads to the development of the concept of relative stability conditions in this paper. We can now formulate the question concerning the relation between relative stability and the solvability of dHYM equation for any smooth projective complex variety:
\begin{quessub}
\label{main-question}
Let $X$ be a smooth projective complex variety. Given an ample $\R$-divisor $\omega$ and a $\R$-divisor $B$. Determine admissible subcategories $0\ne\mathscr{D}_1$ of $\mathscr{D}^b(X)$ for which a line bundle $\mathcal{L}\in\mathrm{Coh}(X)$ is relatively stable with respect to $\sigma_{\r,\omega,B}=(Z_{\omega,B},\mathscr{P}_1:=\mathscr{P}_{\omega,B}\cap\mathscr{D}_1)\in\Stab(X,\mathscr{D}_1)$
if and only if $\cL\in\mathscr{P}_1((0,1])$ admits a dHYM metric with respect to $(\omega, B)$.
\end{quessub}

In the subsequent discussions, we will examine Question \ref{main-question} specifically in the context of projective curves and surfaces. Note that according to Lemma \ref{rel-geo}, for projective surfaces, relative stability conditions are uniquely determined by their relative central charges up to shifting the slicing by $[2n]$. We begin by considering the following trivial cases:

\begin{propsub}
Suppose $X$ is either a smooth projective complex curve with positive genus or a smooth projective complex surface without negative self-intersection curve. Then Question \ref{main-question} is valid for any admissible subcategory.
\end{propsub}
\begin{proof}
Let $0\ne \mathscr{D}_1$ be any admissible subcategory of $\mathscr{D}^b(X)$. If $X$ is a smooth complex curve with positive genus, relative stability on $X$ is equivalent to slope stability, and the dHYM equation reduces to the Hermitian-Einstein equation. Consequently, Question \ref{main-question} is valid for $\mathscr{D}_1$ since every line bundle is slope stable and admits Hermitian-Einstein metric, which is a trivial case of results in \cite{MR0861491}.

In the second case, due to \cite{fu2023deformedhermitianyangmillsflow}, every line bundle on a smooth projective surfaces without negative self-intersection curves admits a dHYM metric. Furthermore, according to \cite{MR3423467}, any line bundle $\mathcal{L}$ is relatively stable with respect to $\sigma_{\omega,0}\in\Stab(X,\mathscr{D}_1)$, for $\omega\in\mathrm{Amp}_{\mathbb R}(X)$.
\end{proof}

We now examine Hirzebruch surfaces, which serve as the first non-trivial example where Bridgeland stability and the existence of the dHYM metric on line bundles are not in agreement. To begin, we review the cone structures and the intersection form on Hirzebruch surfaces through the following lemma.
\begin{lemmasub}\label{lemma: facts on Hirzebruch surface}
Let $X=\mathscr{H}_r$ be the Hirzebruch surface and $\{D_1,D_2,D_3,D_4\}$ be the generators of torus-invariant divisors on $X$. Then
\begin{enumerate}[$(1)$]
\item (\cite[Example 4.1.8]{MR2810322}) The Picard group is generated by $\{D_1,D_2,D_3,D_4\}$ with relations
$$
\begin{aligned}
0\sim\div(\chi^{e_1})&=-D_1+D_3\\
0\sim\div(\chi^{e_2})&=rD_1+D_2-D_4.
\end{aligned}
$$
\item (\cite[Proposition 4.3.3]{MR2810322}) The effective cone of $X$ is given by
$$
\Eff(X)=\{aD_1+bD_2\mid a,b\ge0\}.
$$
\item (\cite[Example 6.1.17]{MR2810322}) The ample cone of $X$ is given by 
$$
\Amp(X)=\{\alpha D_1+\beta D_4\mid \alpha,\beta>0\}.
$$
\item (\cite[Example 6.3.6]{MR2810322}) The intersection matrix of $D_1$ and $D_2$ is given by
$$
\begin{pmatrix}
0&1\\
1&-r
\end{pmatrix}.
$$
\end{enumerate}
\end{lemmasub}

Based on Lemma \ref{lemma: facts on Hirzebruch surface}, we can reduce the testing process based on the twisted ampleness criterion in Theorem \ref{ample-criterion} as follows:

\begin{lemmasub}\label{lemma: twisted ampleness criterion on Hirzebruch surface}
Let $X=\mathscr{H}_r$ be Hirzebruch surface. Then for the twisted ampleness criterion in Theorem \ref{ample-criterion} on $X$, it suffices to test the only negative self-intersection curve. 
\end{lemmasub}
\begin{proof}
Let $\omega=\alpha D_1+\beta D_4$ be an ample $\R$-divisor on $X$ and $\mathcal{L}=kD_3+\ell D_4$ be a line bundle such that $\omega\cdot\ch_1(\mathcal{L})>0$. For any curve $C\subseteq X$, the twisted ampleness criterion for $C$ can be reformulated as
\begin{equation}\label{eq: twisted ampleness criterion on Hirzebruch surface1}
\left(C\cdot\ch_1(\mathcal{L})\right)\left(\omega\cdot\ch_1(\mathcal{L})\right)>\left(\ch_2(\mathcal{L})-\frac12\omega^2\right)\left(C\cdot\omega\right).
\end{equation}
Let $C=D_1$, and then by Lemma \ref{lemma: facts on Hirzebruch surface} the Equation (\ref{eq: twisted ampleness criterion on Hirzebruch surface1}) can be written as 
$$
\ell(\alpha\ell+\beta k+r\beta\ell)>\frac12(2k\ell+r\ell^2-2\alpha\beta-r\beta^2)\beta,
$$
which is equivalent to 
\begin{equation}\label{eq: twisted ampleness criterion on Hirzebruch surface2}
\alpha\ell^2+\alpha\beta^2+\frac12(r\beta\ell^2+r\beta^3)>0.
\end{equation}
It is clear that Equation (\ref{eq: twisted ampleness criterion on Hirzebruch surface2}) holds for any $k,\ell\in\Z$ since $\alpha,\beta>0$ and thus it concludes the proof, as the twisted ampleness criterion is linear with respect to the intersection with $C$.
\end{proof}

\begin{propsub}\label{prop: possible admissible subcategory for question on Hirzebruch surface}
Let $X=\mathscr{H}_r$ be the Hirzebruch surface and $\omega$ be an ample $\R$-divisor on $X$. Then the following admissible subcategories in $\mathscr{D}^b(X)$, which are generated by exceptional objects, satisfy the requirements in Question \ref{main-question}:
\begin{enumerate}[$(1)$]
\item $\mathscr{D}_0=\langle\mathcal{L}\rangle$, where $\mathcal{L}$ is a line in the heart of $\sigma_{\omega,0}$ and admits a dHYM metric.
\item $\mathscr{D}_0=\langle\mathcal{O}_X(\ell D_4),\mathcal{O}_X(D_1+\ell D_4)\rangle$, where $\ell^2<2s\alpha^2/r$.
\end{enumerate}
\end{propsub}
\begin{proof}
For convenience we write $\omega=\alpha(D_1+s D_4)$, where $\alpha,s>0$. By Proposition \ref{prop: surf-with-one}, every line bundle $\mathcal{L}$ in the heart of $\sigma_{\omega,0}$ and admits a dHYM metric with respect to $\omega$ is $\sigma_{\omega,0}$-stable. Thus for the  the admissible subcategory $\mathscr{D}_0=\langle\mathcal{L}\rangle$, the only line bundle is $\mathcal{L}$ itself, and thus $\mathscr{D}_0$ satisfies the requirements in Question \ref{main-question}.

For case (2), consider a line bundle $\mathcal{L}=\mathcal{O}_X(kD_1+\ell D_2)$, where $k,\ell\in\Z$. The twisted ampleness criterion asserts that $\mathcal{L}$ admits a dHYM metric if and only if
\begin{equation}\label{eq: twisted ampleness criterion on Hirzebruch surface}
f_{\omega}(k,\ell):=sk^2+rs\ell k+\frac12\left((s^2r+2s)\alpha^2-r\ell^2\right)>0.
\end{equation}
On the other hand, according to the results in \cite{MR3423467}, $\mathcal{L}$ is $\sigma_{\omega,0}$-stable if and only if
\begin{equation}\label{eq: Bridgeland stability on Hirzebruch surface}
g_{\omega}(k,\ell):=k^2+r\ell k+\frac12\left((r+2)\alpha^2-r\ell^2\right)+\frac{r}{2}(sk+\ell+sr\ell)>0.
\end{equation}

For any fixed $\ell\in\Z$, the equation $f_{\omega}(k,\ell)=0$ admits solutions in $\R$ if and only if
$$
\Delta=s(sr+2)(r\ell^2-2s\alpha^2)\ge0.
$$
Thus for any $k\in\Z$, the line $\mathcal{L}=\mathcal{O}_X(kD_1+\ell D_2)$ which is in the heart of $\sigma_{\omega,0}$ is $\sigma_{\omega,0}$-Bridgeland stable if $\ell^2<2s\alpha^2/r$. For the admissible subcategory $\mathscr{D}_0=\langle\mathcal{O}_X(\ell D_4),\mathcal{O}_X(D_1+\ell D_4)\rangle$, the line bundles in $\mathscr{D}_0$ are $\{\mathcal{O}_X(kD_1+\ell D_4)\mid k\in\Z\}$, and all of them admit dHYM metric and also $\sigma_{\omega,0}$-stable. 
\end{proof}
\begin{remksub}
As conjectured in \cite[Question 3.3]{collins2023stabilitylinebundlesdeformed}, the twisted ampleness criterion should always implies the Bridgeland stability on any smooth projective complex surface. If so, then the above constructions can be generalized to any smooth projective complex surface by finding some special exceptional collections among line bundles which admit dHYM metric. 
\end{remksub}
The subsequent result demonstrates that the counter-example presented in \cite{MR4479717} is a general phenomenon. Moreover, it shows that in order to determine admissible subcategories that address Question \ref{main-question}, it is essential to consider relative stability conditions.

\begin{propsub}\label{prop: counter-example for large volume case}
Let $X=\mathscr{H}_r$ be the Hirzebruch surface. For any ample $\R$-divisor $\omega$ on $X$, there always exists a $\sigma_{\omega,0}$-stable line bundle in the heart that does not admit a dHYM metric with respect to $\omega$.
\end{propsub}
\begin{proof}
For convenience, we still use the notations in the proof of Proposition \ref{prop: possible admissible subcategory for question on Hirzebruch surface}, and from the proof of Proposition \ref{prop: possible admissible subcategory for question on Hirzebruch surface}, it suffices to consider line bundles $\mathcal{L}=\mathcal{O}_X(kD_1+\ell D_2)$ with $\ell^2\ge2s\alpha^2/r$. In this case, the two roots of $f_{\omega}(k,\ell)=0$ in $\R$ are given by
$$
\begin{aligned}
k_{\ell}^+&=\frac12(-r\ell+\sqrt{(r+\frac{2}{s})(r\ell^2-2s\alpha^2)})\\
k_{\ell}^-&=\frac12(-r\ell-\sqrt{(r+\frac{2}{s})(r\ell^2-2s\alpha^2)}).
\end{aligned}
$$
In particular, we have $f_{\omega}(\lfloor k^+_{\ell}\rfloor,\ell)<0$, where $\lfloor k^+_{\ell}\rfloor$ denotes the greatest integer less than or equal to $k^+_{\ell}$. The strategy is to demonstrate that there exists some $\ell\in\Z$ such that
$$
\begin{aligned}
g_{\omega}(\lfloor k^+_{\ell}\rfloor,\ell)>0,
\end{aligned}
$$ 
and consequently, $\mathcal{L}=\mathcal{O}_X(\lfloor k^+_{\ell}\rfloor D_1+\ell D_2)$ is the line bundle which is $\sigma_{\omega,0}$-stable but does not admit a dHYM metric with respect to $\omega$.

For convenience, we denote $k^+_{\ell}-\lfloor k^+_{\ell}\rfloor=\epsilon_{\ell}\in[0,1)$. A direct computation yields
$$
g_{\omega}(\lfloor k^+_{\ell}\rfloor,\ell)=(\frac{r^2}{4}+\frac{r}{2s}+\epsilon_{\ell}-\frac{\epsilon_{\ell}}{s})s\ell+(\frac{r}{4}-\epsilon_{\ell})s\sqrt{s(sr+2)(r\ell^2-2s\alpha^2)}+\left(\epsilon_{\ell}^2-\frac{r}{2}\epsilon_{\ell}\right)s.
$$
\begin{enumerate}[$(1)$]
\item If $r\ge4$, then $r/4-\epsilon_{\ell}\ge0$ always holds, which implies that $g_{\omega}(\lfloor k^+_{\ell}\rfloor,\ell)>0$ for sufficiently large $\ell>0$. 

\item If $1\leq r<4$, then 
$$
\epsilon_{\ell+N}-\epsilon_{\ell}\equiv\frac{-Nr}{2}+\sum_{i=\ell}^{\ell+N}\sqrt{r(r+\frac{2}{s})}\left\{\sqrt{(i+1)^2-\frac{2s\alpha^2}{r}}-\sqrt{i^2-\frac{2s\alpha^2}{r}}\right\}\pmod{\Z},
$$
where $\{\mbox{-}\}$ is the fractional part. Note that the series
$$
\sum_{i=\ell}^{\infty}\sqrt{r(r+\frac{2}{s})}\left\{\sqrt{(i+1)^2-\frac{2s\alpha^2}{r}}-\sqrt{i^2-\frac{2s\alpha^2}{r}}\right\}
$$
diverges, since
$$
\left\{\sqrt{(\ell+1)^2-\frac{2s\alpha^2}{r}}-\sqrt{\ell^2-\frac{2s\alpha^2}{r}}\right\}\sim\frac{1}{2\ell}
$$
as $\ell\to\infty$. Therefore, for any $\ell$, if $\epsilon_{\ell}\leq r/4$, the result is required. Otherwise we can always choose a suitable $N\in\Z_{>0}$ due to the divergence of the series, such that
$$
\epsilon_{\ell+N}<\frac{r}{4}.
$$
As a result, it follows $g_{\omega}(\lfloor k^+_{\ell+N}\rfloor,\ell+N)>0$.
\end{enumerate}
\end{proof}

%=========================================================
%=========================================================
\bibliographystyle{alpha}
\bibliography{sample}

\end{document}